\documentclass[11pt]{article}

\usepackage{amssymb}
\usepackage{mathrsfs}
\usepackage{latexsym,amsmath,amssymb,amsfonts,epsfig,graphicx,cite,psfrag}
\usepackage{eepic,color,colordvi,amscd}
\usepackage{ebezier}
\usepackage{verbatim}
\usepackage{algorithm}
\usepackage{algpseudocode}
\usepackage{tikz}
\usepackage{mathtools}
\usepackage{enumitem}
\usepackage[normalem]{ulem}
\usepackage{longtable}
\usepackage{booktabs}
\usepackage{multirow}
\usepackage{caption}
\usepackage[labelfont=sf]{subcaption}
\usepackage{comment}
\usepackage{xcolor}
\usepackage{extarrows}
\usepackage{float}
\usepackage{amsthm}
\usepackage{epstopdf}
\usepackage{capt-of}
\usepackage{amsopn}
\usepackage{hyperref}
\usetikzlibrary{positioning,arrows.meta,decorations.pathreplacing}
\definecolor{blue}{rgb}{0,0,0.9}
\definecolor{red}{rgb}{0.9,0,0}
\definecolor{green}{rgb}{0,0.9,0}

\DeclareMathOperator{\diag}{diag}
\DeclareMathOperator{\Diag}{Diag}
\DeclareMathOperator*{\argmin}{arg\,min}

\def\Tr{\mathrm{Tr}}

\def\<{\big\langle}
\def\>{\big\rangle}

\def\D{\mathcal{D}}

\def\A{\mathcal{A}}

\def\B{\mathcal{B}}

\def\K{\mathcal{K}}
\def\F{\mathcal{F}}
\def\Z{\mathcal{Z}}

\def\R{\mathbb{R}}

\def\S{\mathbb{S}}

\def\lin{{\operatorname{lin}}}

\def\tb#1{\textbf{#1}}

\def\Range{{\operatorname{Range}}}

\def\tol{\mathrm{tol}}
\def\timelimit{\mathrm{TimeLimit}}

\def\bQ{\overline{Q}}
\newcommand{\xdownarrow}[1]{{\left\downarrow\vbox to #1{}\right.\kern-\nulldelimiterspace}}

\theoremstyle{plain}
\newtheorem{remark}{Remark}
\newtheorem{example}{Example}[section]

\newtheorem{definition}{Definition}
\newtheorem{proposition}{Proposition}
\newtheorem{theorem}{Theorem}

\newtheorem{lemma}{Lemma}

\newcommand{\authornote}[1]{\parbox[t]{0.96\textwidth}{\raggedright #1}}

\title{Sparsity-Cone SDP Relaxations and Applications to Variable Fixing for Sparse Quadratic Programs}

\author{
Di Hou\thanks{\authornote{Department of Mathematics, National University of Singapore, Singapore 119076 (\texttt{dihou@u.nus.edu}).}}
\quad
Thai P.D. Nguyen\thanks{\authornote{Department of Mathematics, National University of Singapore, Singapore 119076 (\texttt{thai.npd@u.nus.edu}).}}
\quad
Kim-Chuan Toh\thanks{\authornote{Department of Mathematics, and Institute of Operations Research and Analytics, National University of Singapore, Singapore 119076 (\texttt{mattohkc@nus.edu.sg}).}}
\quad
Guanyi Wang\thanks{\authornote{Department of Industrial Systems Engineering and Management, National University of Singapore (\texttt{guanyi.w@nus.edu.sg}).}}
}

\date{\today}

\begin{document}

\maketitle

\begin{abstract}
Quadratic programs (QPs) with sparsity constraint are generally NP-hard, and their efficient global solution depends crucially on tractable tight convex relaxations. In this paper, we propose a sparsity-cone semidefinite programming (SC-SDP) relaxation for sparse (indefinite) QPs. Unlike standard SDP liftings, such as the SDP--RLT relaxation, which involve a $(2n+1)$-dimensional semidefinite matrix, the proposed SC-SDP formulation uses only a $(n+1)$-dimensional matrix together with a single sparsity-cone constraint $\mathcal{K}$ to handle the relaxation of the $\ell_0$-norm constraint. We prove that SC-SDP is equivalent in strength to the SDP--RLT relaxation.
We further study the sparsity cone $\mathcal K$, deriving structural characterizations and showing that projection onto $\mathcal K$ can be computed efficiently via a one-dimensional subproblem. Building on the dual of SC-SDP, we derive explicit presolving mechanisms, including a dual-fixing rule for individual variables, a screening-cut rule for excluding larger support patterns, and a dual-refinement step for improving presolving certificates. To solve the resulting relaxation SC-SDP efficiently, we develop a two-phase Riemannian-based augmented Lagrangian method and exploits the structured projection subproblems. Numerical experiments on several classes of sparse QPs show that SC-SDP preserves the bound quality of SDP--RLT while offering substantial computational advantages and practically effective presolving capabilities.
\end{abstract}

\noindent\textbf{Keywords.} sparse quadratic programming, sparsity-constrained optimization, semidefinite programming, sparsity-cone relaxation, variable fixing and screening

\bigskip

\section{Introduction}

\subsection{Sparsity-constrained quadratic programs}

In this paper, we consider the following sparsity-constrained quadratic program (QP):
\begingroup
\renewcommand{\theequation}{P}
\begin{equation}\label{eq:QP}
v^{\mathrm{P}}
\;\coloneqq\;
\min_{x}
\left\{
 x^\top Q x + 2c^\top x
\;\middle|\;
Ax=b,\; Bx\geq d,\;
\|x\|_0\le k
\right\},\tag{P}
\end{equation}
\endgroup
where \(Q \in \mathbb{S}^n\), \(c \in \mathbb{R}^n\), \(A \in \mathbb{R}^{m \times n}\), \(b \in \mathbb{R}^m\), \(B \in \mathbb{R}^{l \times n}\), \(d \in \mathbb{R}^l\), $k$ is a given integer satisfying $0 < k < n$, and \(\|x\|_0\) denotes the number of nonzero entries of \(x\). We assume that the matrix $A$ has full row rank, the feasible region of \eqref{eq:QP} is nonempty, and the objective function of \eqref{eq:QP} is bounded from below over its feasible set.
Problem \eqref{eq:QP} is more general than many sparse QPs considered in the literature because it includes both equality and inequality constraints in addition to the $\ell_0$ sparsity constraint. 
The objective function contains both quadratic and linear terms, enabling it to model a wide range of nonconvex optimization problems. In particular, the matrix $Q$ is allowed to be indefinite.
The $ \ell_0$-norm sparsity constraint also arises in various applications, such as sparse regression, compressed sensing, portfolio optimization, and feature selection. 

\eqref{eq:QP} is generally nonconvex and NP-hard due to the presence of the intractable sparsity constraint and nonconvex quadratic objective. 
To handle the sparsity constraint more effectively, 
one common approach is to
introduce additional $n$ binary variables to represent the support of $x$ and
reformulate \eqref{eq:QP} as the following mixed-binary QP: 
\begin{equation}\label{eq:MIQP}
\min_{x,z}
\left\{
 x^\top Q x + 2c^\top x
\;\middle|\;
Ax=b,\; Bx\geq d,\;
x\circ (1-z)=0,\;
e^\top z\leq k,\;
z\in \{0,1\}^n
\right\}.
\end{equation}
Problem \eqref{eq:MIQP} can then be solved globally by branch-and-bound methods \cite{land1960automatic,mitten1970branch,boyd2007branch}. At each node of the branch-and-bound tree, one solves a relaxation problem of \eqref{eq:MIQP} to obtain a lower bound. The problem is solved when the gap between the lower bound and the best known upper bound is sufficiently small. 
Consequently, it is important to develop relaxations of \eqref{eq:MIQP} that are both tight and computationally efficient. In this paper, we focus on SDP relaxations, which are known to provide tight lower bounds for \eqref{eq:MIQP} \cite{bomze2025tighter}.

\begin{remark}\label{rmk:generalization}
    The relaxation and algorithm proposed in this paper are not restricted to \eqref{eq:QP}. They also extend to sparse quadratically constrained quadratic programs (QCQPs), by directly incorporating the associated lifted constraints into the SDP relaxation \eqref{eq:intro-Decomposed-SDP} and modifying the definition of $\B(Y)$ accordingly in the subsequent subsection. In this extension, the quadratic constraint matrices may also be indefinite. For simplicity, we restrict our attention in this paper to linear constraints.
\end{remark}

\subsection{SC-SDP relaxation}\label{subsec:SDP-RLT-formulation}

The choice of relaxation is crucial for the efficiency of the branch-and-bound algorithm. A tight relaxation can significantly reduce the size of the search tree and thus improve the computational efficiency. 
Therefore, it is important to develop relaxations of \eqref{eq:MIQP} that are both tight and computationally efficient.

\medskip
\noindent\tb{SDP--RLT relaxation.} 
A commonly used convex relaxation of \eqref{eq:MIQP} is the SDP relaxation \cite{shor1990dual}, which is obtained by replacing the quadratic term with a matrix and introducing a positive semidefinite constraint. The explicit formulation of the SDP relaxation \eqref{eq:Shor} of \eqref{eq:MIQP} is given in Subsection \ref{subsec:relaxation-formulation}. 
Although the SDP relaxation can be tight under certain conditions \cite{burer2020exact,wang2022tightness}, it typically provides weak lower bounds for most instances of \eqref{eq:MIQP}. In some cases, the bound can even be negative infinity.

To strengthen the SDP relaxation and improve the bound quality, we incorporate constraints from the reformulation-linearization technique (RLT) \cite{sherali2007rlt,RiNNALplus} and obtain the SDP--RLT relaxation, which typically provides a significantly tighter bound than the standard SDP relaxation alone. However, for large-scale problems, it becomes computationally prohibitive because it must lift both the continuous variables $x$ and the support variables $z$, leading to a semidefinite matrix indexed by $[1;x;z]$ of dimension $2n+1$, together with a large number of RLT constraints.
{For sparse QPs, the $(2n+1)$-dimensional SDP--RLT relaxation can be reduced to an $(n+1)$-dimensional SDP relaxation with $n$ additional SOCP constraints. To the best of our knowledge, however, these SOCP constraints do not readily yield explicit presolving rules for variable fixing.}

\medskip

\noindent\tb{SC-SDP relaxation.} 
{To reduce the computational burden of the SDP--RLT relaxation of \eqref{eq:MIQP} and eliminate the explicit support variables $z$ and the associated SOCP constraints in its reduced form,} we propose a sparsity-cone SDP relaxation. Standard SDP formulations for \eqref{eq:MIQP} typically employ a lifted matrix involving both the $x$- and $z$-variables, thereby effectively doubling the dimension of the lifted space. Our objective is to construct a new SDP relaxation of comparable strength without this dimensional increase. The compact relaxation keeps only the lifted matrix associated with $[1;x]$
\begin{equation}\label{eq:intro-Y}
    Y
\;\coloneq\; 
\begin{bmatrix}
    Y_{11} &Y_{12}\\Y_{21} & Y_{22}
\end{bmatrix}
\;=\;
\begin{bmatrix}
    1&x^\top \\
    x&X
\end{bmatrix}\succeq 0,
\end{equation}
as its semidefinite variable. Using this matrix, the RLT constraints derived from the original linear constraints can be written compactly as
\begin{equation}\label{eq:intro-RLT}
    \A(Y)\coloneqq\begin{bmatrix}
    -b&A
    \end{bmatrix}
    Y\begin{bmatrix}
        -b^\top\\A^\top
    \end{bmatrix}=0,\quad
    \B(Y)\coloneqq  \begin{bmatrix}
     1&0
        \\-d&B
    \end{bmatrix}
    Y\begin{bmatrix}
        1&0
        \\-d&B
    \end{bmatrix}^\top\geq 0.
\end{equation}
These are exactly the RLT constraints obtained from the original linear constraints in the $[1;x]$ lifting. {To replace the support variables and SOCP constraints in the reduced SDP--RLT form while preserving the sparsity information}, we introduce the sparsity-cone
\begin{equation}\label{eq:intro-K-set}
    \K\coloneq \left\{Y\in\S^{n+1}: \D(Y)\coloneq \begin{bmatrix}
    k Y_{11}&Y_{12}\\Y_{21} & \Diag(Y_{22})
\end{bmatrix}\succeq 0\right\}.
\end{equation}
The cone $\K$ encodes the sparsity structure of the original problem. More precisely, it is exact on normalized rank-one lifts: for $Y=[1;x][1;x]^\top$, the membership condition $Y\in\K$ is equivalent to the cardinality bound $\|x\|_0\le k$; see Theorem~\ref{thm:equivalence-K}.
Combining the RLT constraints \eqref{eq:intro-RLT} with the cone constraint, we obtain the following sparsity-cone semidefinite (SC-SDP) relaxation of \eqref{eq:QP}:
\begingroup
\renewcommand{\theequation}{Q}
\begin{equation}
\label{eq:intro-Decomposed-SDP}
\min_{Y}\left\{
\langle \bQ,Y \rangle  \;\middle|\;  
 \A(Y)=0,\; \B(Y)\geq 0,\; 
Y_{11}=1,\;
 Y\in \K\cap \S^{n+1}_+
\right\},\tag{Q}
\end{equation}
\endgroup
where the cost matrix $\bQ\coloneqq [ 0, c^\top; c, Q]$. 
Compared with the SDP--RLT relaxation of \eqref{eq:MIQP}, 
our new relaxation \eqref{eq:intro-Decomposed-SDP} has several key advantages:
\begin{enumerate}[itemsep=2pt, topsep=4pt]
    \item \eqref{eq:intro-Decomposed-SDP} preserves the relaxation strength of the SDP--RLT relaxation of \eqref{eq:MIQP}. Indeed, Theorem~\ref{thm:equivalence} shows that the two relaxations are equivalent, so \eqref{eq:intro-Decomposed-SDP} yields exactly the same lower bound for \eqref{eq:MIQP} without any loss in bound quality.
    \item \eqref{eq:intro-Decomposed-SDP} is substantially more compact. It uses a single semidefinite matrix of size $n+1$ together with the cone constraint $Y\in\K$, whereas the SDP--RLT relaxation of \eqref{eq:MIQP} uses a semidefinite matrix of size $2n+1$ together with the RLT constraints in \eqref{eq:intro-RLT}, as well as the 
    additional constraints that encode the sparsity constraint based on the variable $z$. 
    This reduction leads to a cleaner formulation and significantly lowers the lifted dimension. Numerically, SC-SDP also tends to yield much lower-rank solutions than the SDP--RLT model, which further supports its computational advantage.
    \item The cone $\K$ in \eqref{eq:intro-K-set} is not only a modeling device but also computationally tractable. As shown in Section~\ref{sec:cone-geometry}, the projection onto $\K$ admits a reduction to a one-dimensional subproblem. Moreover, the Moreau decomposition allows the algorithm to recover the multiplier update associated with $\K^*$ from this projection. This structure is essential for the first-order method developed in Section~\ref{sec:algorithm} to efficiently solve the SC-SDP 
    relaxation.
    \item Another important advantage of \eqref{eq:intro-Decomposed-SDP}
    is that its dual has a simple structure that allows us to derive explicit presolving rules for the original mixed-binary formulation \eqref{eq:MIQP}. In particular, we show that the dual solution of \eqref{eq:intro-Decomposed-SDP} yields a dual-fixing rule for individual variables, a screening-cut rule for excluding larger support patterns. Moreover, based on the dual solution, we can develop
    a dual-refinement step for improving the fixing and screening certificates. To the best of our knowledge, this provides the first explicit fixing mechanism derived from an SDP relaxation for general sparsity-constrained quadratic programs with linear constraints.
\end{enumerate}

\subsection{Presolving with SC-SDP}

For sparsity-constrained problems such as \eqref{eq:QP} and \eqref{eq:MIQP}, the main combinatorial difficulty is to identify the support of the solution. This makes presolving especially valuable in branch-and-bound methods: fixing variables, or more generally ruling out 
certain support patterns before branching, can substantially reduce the search tree. Accordingly, a useful relaxation such as \eqref{eq:intro-Decomposed-SDP} should do more than provide a strong lower bound; it should also reveal structural information that can be turned into explicit fixing and screening certificates.

\medskip
\noindent\tb{Related work.}
Classical presolving methods based on LP relaxations \cite{achterberg2020presolve,gamrath2015progress} provide a rich toolbox of fixing techniques, but their effectiveness is limited by the weakness of the underlying relaxations. Stronger relaxations have therefore been explored in more specialized settings. In sparse ridge regression, perspective-based methods are substantially stronger and have produced effective fixing and screening rules \cite{atamturk2020safe,tan2025screening}, but they remain restrictive.
First, they typically require strong convexity of the objective function with $Q\succ 0$, and their improvement depends on the strength of that convexity. Second, they essentially apply only to problems without additional constraints. In contrast, our SDP-based relaxations such as \eqref{eq:intro-Decomposed-SDP} apply to more general sparsity-constrained quadratic problems 
with additional linear equality/inequality constraints without
requiring $Q$ to be positive definite, and they
can provide strong lower bounds. 
However, effective fixing rules derived from such SDP relaxations are 
currently still unexplored. Although related questions have been studied in the mixed-integer SDP literature \cite{gally2018framework,matter2023presolving}, presolving remains difficult there as well, and the resulting techniques do not directly extend to the explicit fixing rules developed here.

\medskip

The special structure of the SC-SDP relaxation \eqref{eq:intro-Decomposed-SDP} helps bridge this gap. In Appendix~\ref{app:support-restrictions}, we explain how SC-SDP can be used under support restrictions, with strong branching as a special case. Subsection~\ref{subsec:variable-fixing}
then develops a dual-fixing rule for individual variables, Subsection~\ref{subsec:screening-cuts} derives screening cuts for larger support patterns, and Subsection~\ref{subsec:dual-refinement} introduces a dual-refinement step for strengthening the resulting fixing and screening certificates. 
We briefly explain the fixing and screening rules below. 

\medskip
\noindent\tb{Dual fixing.}
Let $z$ be the binary vector in \eqref{eq:MIQP}, and let $z_p^*$ be the $p$-th component of an optimal solution of \eqref{eq:MIQP}. Given a dual solution of \eqref{eq:dual-Decomposed-SDP}, define the scores $w_i$ as in Theorem~\ref{thm:dual-fixing}, and let $w_{[i]}$ denote the $i$-th largest value in $\{w_1,\dots,w_n\}$. If $v^{\mathrm{UB}}$ is an upper bound on the optimal value and $v^{\mathrm{LB}}$ is the lower bound produced by \eqref{eq:intro-Decomposed-SDP}, then Theorem~\ref{thm:dual-fixing} yields the rule
\[
\begin{aligned}
z_p^*=0
&\quad\text{if }w_{[k]}-w_p>v^{\mathrm{UB}}-v^{\mathrm{LB}},\\
z_p^*=1
&\quad\text{if }w_p-w_{[k+1]}>v^{\mathrm{UB}}-v^{\mathrm{LB}}.
\end{aligned}
\]

\medskip
\noindent\tb{Screening cuts.}
When single-variable fixing is not sufficient, the same dual certificate can be used to rule out larger support patterns. For an SCG-tuple $(S,N,C)$ defined in Definition~\ref{def:scg-tuple}, where $S$ and $N$ are candidate support and non-support sets and $C$ is the associated complement support set, Theorem~\ref{thm:screening-cuts} gives
\begin{equation*}
\sum_{i=1}^k w_{[i]}-\sum_{i\in S}w_i-\sum_{i\in C}w_i
>v^{\mathrm{UB}}-v^{\mathrm{LB}}
\quad \Longrightarrow\quad
\sum_{i\in S} z_i+\sum_{i\in N}(1-z_i)
\le |S|+|N|-1.
\end{equation*}

\medskip
\noindent\tb{Dual refinement.}
Because the optimal dual solution of \eqref{eq:intro-Decomposed-SDP} is generally not unique, different optimal solutions may yield the same lower bound but different fixing and screening performance. We therefore introduce a dual-refinement step, which keeps the SC-SDP bound unchanged while selecting a more effective presolving certificate, and derive the compact $(3n+1)$-variable semidefinite--second-order-cone representable formulation \eqref{eq:dual-refinement-uqp-reduced} by using the arrow representation of $\D^*(W)$.

\medskip

The presolving rules derived from the SC-SDP relaxation are explicit and practical, and they can be applied at the root node or throughout the branch-and-bound tree to substantially reduce the search space. To the best of our knowledge, this provides the first explicit fixing mechanism derived from an SDP relaxation for general sparsity-constrained quadratic programs with linear constraints.

\subsection{Low-rank augmented Lagrangian method}

Although the SC-SDP relaxation \eqref{eq:intro-Decomposed-SDP} is substantially more compact than the SDP--RLT formulation, solving it efficiently at scale remains challenging. The main difficulties are as follows:
\begin{enumerate}[itemsep=2pt, topsep=4pt]
    \item A huge number of $\Omega(m^2)$ equality constraints;
    \item A large number of $\Omega(l^2)$ inequality constraints;
    \item A nontrivial conic constraint $Y\in\K$;
    \item Failure of Slater’s condition and strong duality whenever $A\not=0$.
\end{enumerate}
In this paper, we propose RiNNAL-S, a low-rank augmented Lagrangian method (ALM) for solving the SC-SDP relaxation \eqref{eq:intro-Decomposed-SDP}. Our method builds on the RiNNAL framework and its variants developed in \cite{RiNNAL,RiNNALplus,RiNNALPOP}, but extends them to handle the SC-SDP formulation with the new cone constraint $\K$.
More specifically, we rewrite \eqref{eq:intro-Decomposed-SDP} in the splitting form:
\begin{equation}
\label{eq:intro-splitting}
\min_{Y,Z}\left\{
\langle \bQ,Y\rangle+\delta_{\F}(Y)+\delta_{\K}(Z)
\;\middle|\;
\B(Y)\ge 0,\; Y-Z=0
\right\}
\end{equation}
where 
\[
\F=\{Y\in\S^{n+1}_+:\A(Y)=0,\;Y_{11}=1,\;Y\succeq 0\}.
\]
Let $\Lambda\ge 0$ and $W\in\K^*$ be the multipliers associated with $\B(Y)\ge 0$ and $Y-Z=0$, respectively, and let $\sigma>0$ be the penalty parameter. Then each outer ALM iteration solves a subproblem of the form
\small
\begin{equation}
\label{eq:intro-ALM-subproblem}
\min_{Y}\; \left\{
    L_\sigma(Y,\Lambda,W)
\coloneq \langle \bQ,Y\rangle
+\frac{\sigma}{2}\|(\sigma^{-1}\Lambda-\B(Y))_+\|^2
+\frac{\sigma}{2}\|\Pi_{\K^*}(\sigma^{-1}W-Y)\|^2
\;\middle|\;
Y\in\F
\right\}
\end{equation}
\normalsize
followed by the multiplier updates for $\Lambda$ and $W$. This ALM framework guarantees global convergence of the proposed method, so the main remaining issue is how to solve the ALM subproblem \eqref{eq:intro-ALM-subproblem} efficiently by exploiting the structure of $\F$ and $\K$. As shown in Section~\ref{sec:cone-geometry} and Subsection~\ref{subsec:projection-F}, projections onto $\F$ and $\K$ can be computed efficiently by solving one-dimensional subproblems. To solve \eqref{eq:intro-ALM-subproblem} efficiently, we alternately use two complementary phases.

\medskip
\noindent\tb{Low-rank phase.}
Suppose the ALM subproblem \eqref{eq:intro-ALM-subproblem} admits an optimal solution of
rank $r$. To exploit this potential low-rank structure, we apply the following low-rank factorization
\[
Y=\widehat R\widehat R^\top,
\qquad
\widehat R=\begin{bmatrix}e_1^\top\\ R\end{bmatrix},
\qquad
R\in\R^{n\times r},
\]
where the first row can be fixed as $e_1^\top$ so that $Y_{11}=1$ holds. 
Thus, solving \eqref{eq:intro-ALM-subproblem} is equivalent to solving the following nonconvex problem:
\begin{equation}\label{eq:intro-low-rank-subproblem}
\min_{R}\;\left\{ L_\sigma(\widehat R\widehat R^\top
,\Lambda,W)\mid \widehat R\widehat R^\top\in \F \right\}.
\end{equation}
Directly handling the nonconvex quadratic constraint in \eqref{eq:intro-low-rank-subproblem} is difficult, but the special structure of $\F$ allows us to rewrite this constraint as a simple linear constraint $AR=be_1^\top$. The detailed low-rank model and algorithm are given in Subsection~\ref{subsec:low-rank}.

\medskip
\noindent\tb{Convex lifting phase.}
The low-rank phase reduces the objective value of \eqref{eq:intro-low-rank-subproblem}, but does not guarantee global convergence due to its nonconvexity. To restore global convergence, we lift the low-rank solution back to the original convex space \eqref{eq:intro-ALM-subproblem} and perform a projected gradient step
\[
Y^+=\Pi_{\F}\bigl(Y-\eta\nabla_Y L_\sigma(Y,\Lambda,W)\bigr).
\]
On the one hand, this step guarantees global convergence for the ALM subproblem. On the other hand, it automatically identifies a suitable rank and an initial point for the next low-rank phase. The convex lifting procedure is presented in detail in Subsection~\ref{subsec:convex-lifting}.

\subsection{Summary of our contributions}

The main contributions of this paper can be summarized as follows.

\begin{enumerate}[leftmargin=*, itemsep=2pt, topsep=4pt]
    \item We introduce a reduced SDP--RLT relaxation and propose the sparsity-cone semidefinite relaxation \eqref{eq:intro-Decomposed-SDP} for sparsity-constrained quadratic programs \eqref{eq:MIQP}. These formulations retain the tightness of the classical SDP--RLT relaxation while reducing the main lifted matrix dimension from $2n+1$ to $n+1$, with SC-SDP further replacing the explicit support variables and perspective LMIs by a single structured cone constraint $\K$.

    \item We establish the geometric and algorithmic foundations of the sparsity cone $\K$ and its dual cone $\K^*$. In particular, we derive equivalent descriptions and explicit characterizations of these cones and show that projection onto $\K$ admits an efficient one-dimensional reduction. We also use the Moreau decomposition to recover the $\K^*$ multiplier update from the $\K$ projection.

    \item We derive explicit presolving rules from the dual solution of SC-SDP for the original mixed-binary formulation \eqref{eq:MIQP}. In particular, we show that SC-SDP can provide more informative support-restriction bounds, with strong branching as a special case, establish a dual-fixing rule for individual variables and a screening-cut rule for excluding larger support patterns. We also introduce a dual-refinement step for improving the resulting fixing and screening certificates, thereby reducing the search space before branching.

    \item We develop RiNNAL-S, a two-phase augmented Lagrangian method for solving the SC-SDP relaxation. In the first phase, RiNNAL-S solves a low-rank ALM subproblem in factorized form under the affine constraint induced by $\F$. In the second phase, it performs a convex lifting step through a projected gradient update in the full matrix space. The method exploits the structured projection subproblems associated with $\F$ and $\K$, automatically identifies a suitable rank for the low-rank phase and guarantees global convergence for solving the SC-SDP relaxation.
    
    \item We illustrate the scope of the SC-SDP framework in the unconstrained sparse ridge regression setting. In particular, we show that SC-SDP dominates the standard perspective relaxation and matches the optimal perspective relaxation.
    
    \item We evaluate the performance of our \eqref{eq:intro-Decomposed-SDP} framework by
    conducting extensive computational experiments on several classes of $\ell_0$-constrained quadratic optimization problems. The results show that SC-SDP preserves the bound quality of SDP--RLT while delivering significant computational advantages, and often yields much lower-rank solutions in practice. Moreover, 
    \eqref{eq:intro-Decomposed-SDP} can offer strong variable fixing and 
    screening capability for utilization within a branch-and-bound
    solver. 
\end{enumerate}

\subsection{Organization}

The rest of this paper is organized as follows. Section~\ref{sec:cone-geometry} studies the sparsity cone $\K$, the dual cone $\K^*$, and the associated projection problem onto $\K$. Section~\ref{sec:sc-sdp} presents the reduced SDP--RLT and SC-SDP relaxations and studies the perspective equivalence in an unconstrained sparse ridge regression setting. Section~\ref{sec:presolving} develops the presolving framework derived from the SC-SDP dual structure. Section~\ref{sec:algorithm} develops the algorithmic framework for solving SC-SDP. Section~\ref{sec:numerical} reports numerical experiments on several classes of sparse quadratic optimization problems.

\subsection{Notations}

Let $\langle A,B\rangle := \Tr\left(AB^\top\right)$ denote the matrix inner product and $\|\cdot\|$ be its induced Frobenius norm in $\S^n$. Define $e$ as a column vector of all ones, and $e_1$ as a column vector with $1$ as its first entry and zero otherwise. We use $\circ$ to denote the element-wise multiplication operation between two vectors or matrices of the same size. Let $[n]:=\{1,2,\ldots,n\}$ for any positive integer $n$. For $x\in\R^n$, we write $\operatorname{supp}(x):=\{j\in[n]:x_j\neq 0\}$ for its support. For nonnegative scalars $x,t$, the quotient $x/t$ is understood in the extended-real sense: $x/0:=0$ if $x=0$, and $x/0:=+\infty$ if $x>0$. 

For two matrices $U,V$ having the same number of columns, we use
$[U;V]$ to denote the matrix obtained by appending $V$ to the last row
of $U$.
For a matrix $X\in\S^{n+1}$, we denote its block decomposition as follows:
\begin{equation}
X=
\begin{bmatrix}
X_{11} & X_{12}\\
X_{21} & X_{22}
\end{bmatrix}
\in
\begin{bmatrix}
\R & \R^{1\times n}\\
\R^{n\times 1} & \S^n
\end{bmatrix}.
\end{equation}
We write
\begin{equation}
\label{eq:nonnegative-cone}
\mathcal{N}:=\{Y\in\S^{n+1}:Y\ge 0\}
\end{equation}
for the cone of symmetric entrywise nonnegative matrices of size $n+1$.
For any scalar, vector, or matrix $U$, we write $U_+$ for its entrywise
positive part.
For a matrix $R\in\R^{n\times r}$, we let $R_i\in\R^{1\times r}$ be its $i$-th row, and define $\widehat{R}:=[e_1^\top;R]$, which augments $R$ with the first standard basis row vector.

\section{Sparsity cones}\label{sec:cone-geometry}

This section develops the geometric and algorithmic foundations for the sparsity cones $\K$ and $\K^*$. We first study the geometric properties of $\K$ in Subsection~\ref{subsec:property-K}, then derive efficient projection methods onto $\K$ and its nonnegative variant in Subsection~\ref{subsec:projection-K}. We next analyze the dual cone $\K^*$ in Subsection~\ref{subsec:property-Kstar}.

\subsection{Geometric properties of \texorpdfstring{$\K$}{K}}\label{subsec:property-K}

For a lifted matrix $Y=[1,x^\top; x, X]$, the cone $\K$ only couples the first row and column with the diagonal of $X$. This observation leads to a simple structural description of $\K$. In particular,
\begin{equation*}
\lin(\K)
=\K\cap(-\K)=\{Y\in\S^{n+1}:\; Y_{11}=0,\ Y_{12}=0,\ \Diag(Y_{22})=0\},
\end{equation*}
so $\K$ is not pointed. Nevertheless, $\K$ still admits useful characterizations and efficient projection methods.
The following lemma provides several equivalent characterizations of $\K$ for matrices with $Y_{11}=1$.

\begin{lemma}
\label{lemma:equivalence}
The following equivalence holds for any $Y=[1,x^\top; x, X]\in\S^{n+1}_+$:
\begin{equation}\label{eq:kDiag-xx-general}
   Y\in\K
   \;\; \Longleftrightarrow\;\;
    \D(Y)\succeq 0
\;\; \Longleftrightarrow\;\;
k\Diag(X)\succeq xx^\top 
\;\; \Longleftrightarrow\;\;
    \sum_{i=1}^n \frac{x_i^2}{X_{ii}}\leq k.
\end{equation}
Note that for a nonnegative scalar $\alpha$, we define 
$\alpha/0 = 0$ if $\alpha = 0$, and $\alpha/0 = +\infty$ if $\alpha > 0$.
\end{lemma}

\begin{proof}\;
The equivalence $Y\in\K \Longleftrightarrow \D(Y)\succeq 0$ is exactly the definition of $\K$ in \eqref{eq:intro-K-set}. For $Y=[1,x^\top; x, X]$, we have $\D(Y)=[k,x^\top; x, \Diag(X)]$. Because $k>0$, applying the Schur complement to this matrix yields
 $\D(Y)\succeq 0 \Longleftrightarrow \Diag(X)-k^{-1}xx^\top\succeq 0 \Longleftrightarrow k\Diag(X)-xx^\top\succeq 0$.
This proves the second equivalence.
It remains to show that $\D(Y)\succeq 0 \Longleftrightarrow \sum_{i=1}^n x_i^2/X_{ii}\le k$.
Since $Y\succeq 0$, each principal $2\times 2$ submatrix $[1,x_i; x_i, X_{ii}]$ is positive semidefinite. Hence $X_{ii}\ge 0$ for every $i$, and $X_{ii}=0$ implies $x_i=0$. Therefore $x\in\Range(\Diag(X))$. Define $\Diag(X)^\dagger=\Diag(d)$ by setting $d_i=1/X_{ii}$ if $X_{ii}>0$ and $d_i=0$ if $X_{ii}=0$. Then
 $x^\top \Diag(X)^\dagger x = \sum_{i=1}^n x_i^2/X_{ii}$,
where the terms with $X_{ii}=0$ are zero because then $x_i=0$.
Now apply the generalized Schur complement to the matrix $[k,x^\top; x, \Diag(X)]$. This shows that
 $\D(Y)\succeq 0 \Longleftrightarrow x\in\Range(\Diag(X)) \text{ and } x^\top \Diag(X)^\dagger x\le k$.
Since we already know that $x\in\Range(\Diag(X))$, this reduces to the third equivalence.
\end{proof}

The next result shows that, on rank-one lifted points, membership in $\K$ exactly captures the underlying $\ell_0$-sparsity property.
\begin{theorem}
\label{thm:equivalence-K}
Let $Y=yy^\top$ with $y=(t,x)\in\R\times\R^n$. Then $Y\in\K$ if and only if $t=0$ or $\|x\|_0\le k$. In particular, for any $x\in\R^n$, we have $[1; x][1; x]^\top\in\K$ if and only if $\|x\|_0\le k$.
\end{theorem}

\begin{proof}\;
If $t=0$, then $\D(Y)=\Diag(0,x_1^2,\ldots,x_n^2)\succeq 0$, so $Y\in\K$. Now assume that $t\neq 0$ and set $u=x/t$. Since $\K$ is a cone,
\[
Y\in\K
\quad\Longleftrightarrow\quad
t^{-2}Y=
\begin{bmatrix}
1 & u^\top\\
u & uu^\top
\end{bmatrix}\in\K.
\]
Applying Lemma~\ref{lemma:equivalence} to the normalized matrix $t^{-2}Y$, we obtain
\[
\begin{bmatrix}
1 & u^\top\\
u & uu^\top
\end{bmatrix}\in\K
\quad\Longleftrightarrow\quad
\sum_{i=1}^n \frac{u_i^2}{u_i^2}\le k.
\]
The right-hand side is exactly $\|u\|_0=\|x\|_0\le k$. This proves both statements.
\end{proof}

We use the following example to illustrate the geometry of $\K\cap\S^{n+1}_+$ and the relationship between the $\K$-constraint and the semidefinite constraint.
\begin{example}[A slice of $\K\cap\S_+^3$]
Consider the case $n=2$ and $k=1$, and fix
\[
Y=
\begin{bmatrix}
1 & x_1 & x_2\\
x_1 & 1 & X_{12}\\
x_2 & X_{12} & 1
\end{bmatrix}.
\]
Then Lemma~\ref{lemma:equivalence} gives
$
Y\in\K
\;\Longleftrightarrow\;
x_1^2+x_2^2\le 1,
$
so the $\K$-constraint is a cylinder in the $(x_1,x_2,X_{12})$-space. On the other hand,
\[
Y\succeq0
\quad\Longleftrightarrow\quad
X_{12}\in\left[
x_1x_2-\sqrt{(1-x_1^2)(1-x_2^2)},\,
x_1x_2+\sqrt{(1-x_1^2)(1-x_2^2)}
\right].
\]
Hence this slice of $\K\cap\S_+^3$ is obtained by intersecting the cylinder induced by $\K$ with the region between the two semidefinite boundary sheets; see Figure~\ref{fig:geometry-k-cap-s-plus}(\subref{fig:geometry-k-cap-s-plus-a}) for the cylinder and semidefinite boundary, and Figure~\ref{fig:geometry-k-cap-s-plus}(\subref{fig:geometry-k-cap-s-plus-b}) for the resulting intersection region.

\begin{figure}[ht!]
    \centering
    \begin{subfigure}[t]{0.35\textwidth}
        \centering
        \includegraphics[width=\textwidth]{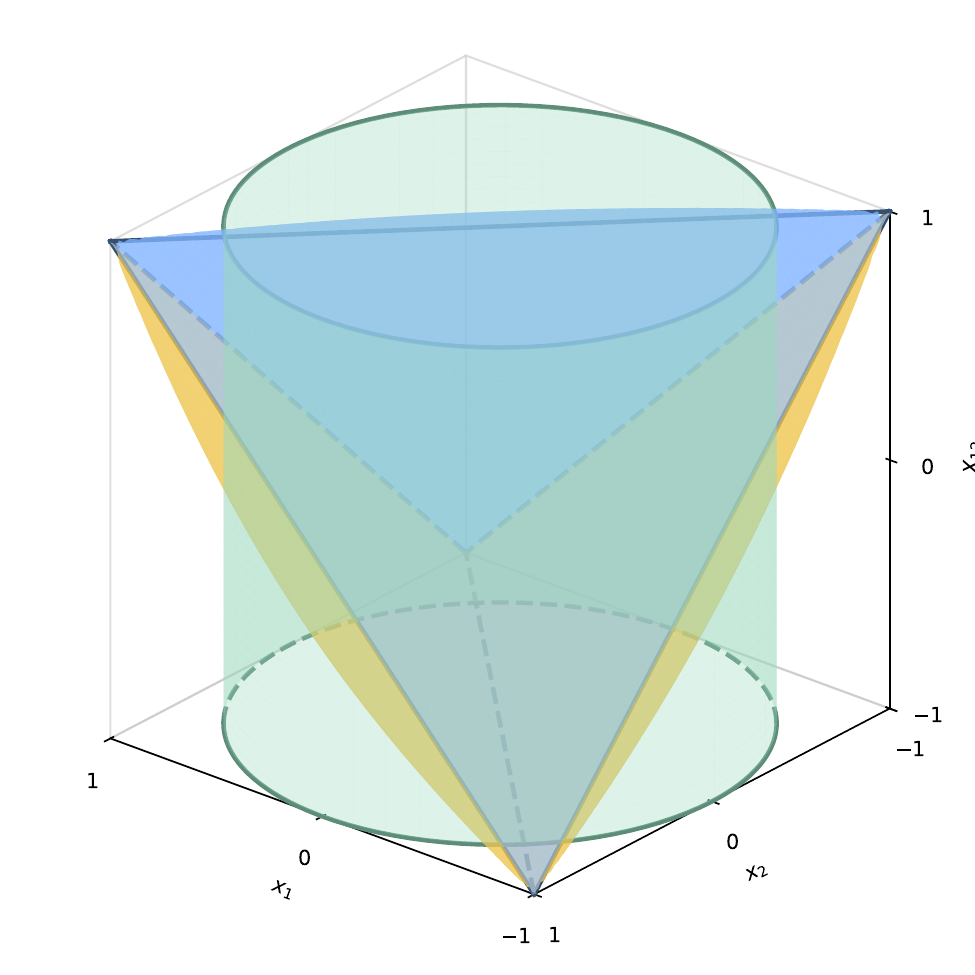}
        \caption{Cylinder and semidefinite boundary.}
        \label{fig:geometry-k-cap-s-plus-a}
    \end{subfigure}
    \qquad\qquad
    \begin{subfigure}[t]{0.35\textwidth}
        \centering
        \includegraphics[width=\textwidth]{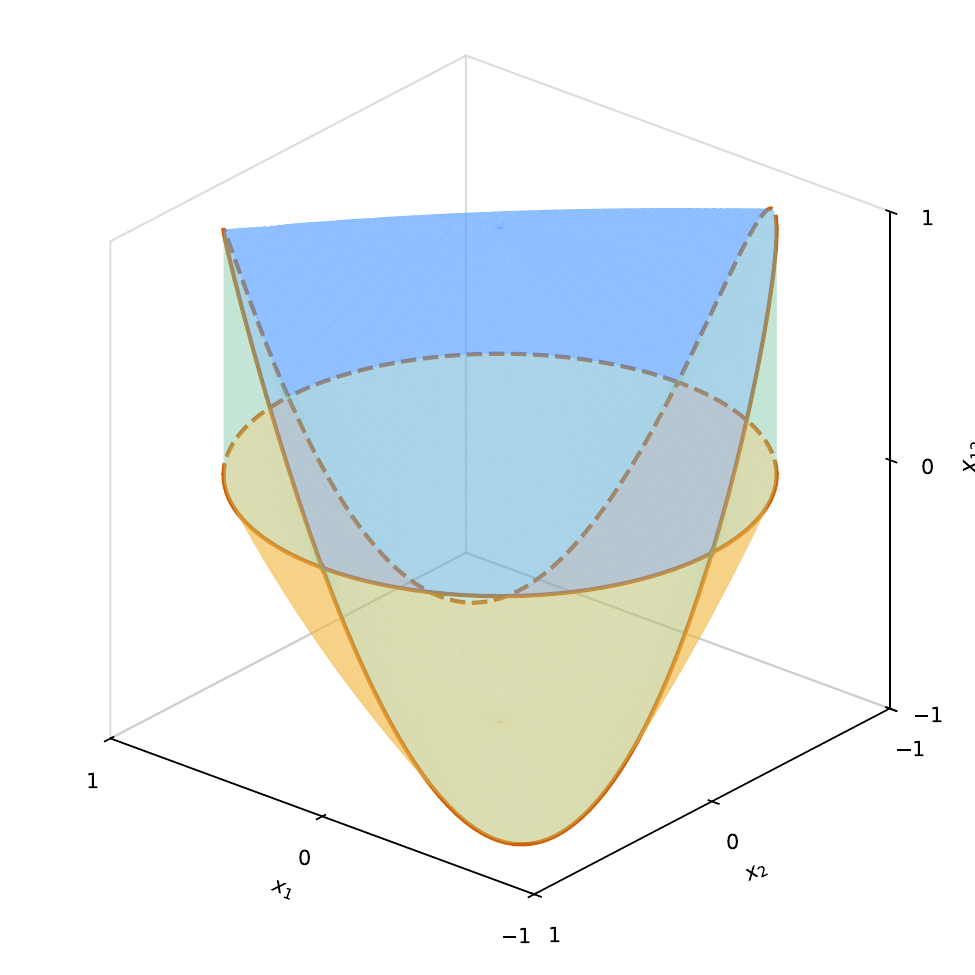}
        \caption{Intersection region.}
        \label{fig:geometry-k-cap-s-plus-b}
    \end{subfigure}
    \caption{Geometry of $\K\cap\S_+^3$ for $n=2$ and $k=1$.}
    \label{fig:geometry-k-cap-s-plus}
\end{figure}
\end{example}

\subsection{Projection onto \texorpdfstring{$\K$}{K}}\label{subsec:projection-K}

Projection onto $\K$ is required repeatedly in the algorithm. Since membership in $\K$ depends only on the first row and column and on the diagonal of the lower-right block, the off-diagonal part of $\bar X$ is preserved by the projection. The next result makes this reduction explicit. For
\[
\bar Y=
\begin{bmatrix}
\bar a & \bar x^\top\\
\bar x & \bar X
\end{bmatrix}\in \S^{n+1},
\qquad
\bar d:=\diag(\bar X),
\]
projection onto $\K$ reduces to a structured convex problem, as shown in 
the next proposition.

\begin{proposition}\label{prop:proj-K-reduced}
The projection of $\bar Y$ onto $\K$ is
\[
\Pi_{\K}(\bar Y)=
\begin{bmatrix}
a^* & (x^*)^\top\\
x^* & \Diag(d^*)+\bigl(\bar X-\Diag(\bar d)\bigr)
\end{bmatrix},
\]
where 
\begin{equation}
\label{eq:proj-K-reduced}
(a^*,x^*,d^*) =
\argmin_{\substack{a\in \R,\ x\in\R^n,\ d \in \R^n}}
\left\{
\frac12(a-\bar a)^2+\|x-\bar x\|^2+\frac12\|d-\bar d\|^2
\;\middle|\;
\begin{bmatrix}
ka & x^\top\\
x & \Diag(d)
\end{bmatrix}\succeq 0
\right\}.
\end{equation}
\end{proposition}

\begin{proof}\;
By \eqref{eq:kDiag-xx-general}, membership in $\K$ depends only on $Y_{11}$, $Y_{12}$, $Y_{21}$, and $\diag(Y_{22})$. Since the Frobenius norm is separable entrywise, the off-diagonal entries of $Y_{22}$ are unchanged by the projection. Hence any minimizer has the displayed form, which yields \eqref{eq:proj-K-reduced}. The feasible set is closed and convex, and the objective is strongly convex, so the solution is unique.
\end{proof}

The reduced problem \eqref{eq:proj-K-reduced} already isolates the essential difficulty. By the Schur complement, it is equivalent to a convex problem with the single nonlinear constraint $\sum_{i=1}^n x_i^2/d_i\le ka$. If the projection onto the orthant $a\ge 0$ and $d\ge 0$ is feasible, then it is optimal. Otherwise, strong duality introduces a single multiplier $\lambda$, and eliminating $x$ reduces the remaining computation to a scalar equation in $\lambda$. The next result gives the resulting one-dimensional reduction.

\begin{proposition}
\label{prop:appendix-proj-K}
Let $(a^*,x^*,d^*)$ be the unique optimal solution of problem~\eqref{eq:proj-K-reduced}.
For any $\lambda >0$, consider for each $i\in[n]$,
\begin{equation} \label{eq:d-lambda}
d_i(\lambda) := \argmin_{d_i\ge 0}\left\{
\frac12(d_i-\bar d_i)^2+\frac{\lambda\bar x_i^2}{d_i+\lambda}
\right\},\qquad i\in[n].
\end{equation}
Then $d_i(\lambda) = 0$ if 
$\bar d_i +{\bar x_i^2}/{\lambda}\leq 0$; otherwise, $d_i(\lambda)>0$ is the unique solution of
\begin{equation}
\label{eq:cubic-K}
(d_i-\bar d_i)(d_i+\lambda)^2=\lambda \bar x_i^2,
\qquad i\in[n].
\end{equation}
We have that either of the following two cases holds:
\begin{itemize}
\item[(i)] If $\sum_{i=1}^n \bar{x}_i^2/(\bar{d}_i)_+ \le k \bar{a}_+$, then
$
(a^*,x^*,d^*)=(\bar{a}_+,\bar{x},\bar{d}_+).
$

\item[(ii)] Otherwise, there exists a multiplier $\lambda^*>0$ such that
\[
a^*=(\bar a+k\lambda^*)_+,
\qquad
x_i^*=\frac{d_i^*}{d_i^*+\lambda^*}\,\bar x_i,\quad i\in[n],
\]
where $d_i^*=d_i(\lambda^*)$.
Such a multiplier $\lambda^*$ satisfies the scalar equation
\begin{equation}
\label{eq:root-K}
\sum_{i=1}^n
\frac{\bar x_i^2\, d_i(\lambda^*)}{(d_i(\lambda^*)+\lambda^*)^2}
=
k(\bar a+k\lambda^*)_+.
\end{equation}
\end{itemize}
\end{proposition}

\begin{proof}\;
By the Schur complement, problem~\eqref{eq:proj-K-reduced} is equivalent to
\begin{equation}
\label{eq:proj-K-reduced-2}
\min_{a,x,d}\left\{
\frac12(a-\bar a)^2+\|x-\bar x\|^2+\frac12\|d-\bar d\|^2
\;\middle|\;
a\ge 0,\ d\ge 0,\ \sum_{i=1}^n \frac{x_i^2}{d_i}\le k a
\right\}.
\end{equation}
Note that $(\bar{a}_+,\bar{x},\bar{d}_+)$ is the projection of $(\bar a,\bar x,\bar d)$ onto the set defined by $a\ge 0$ and $d\ge 0$, which is a superset of the feasible set of \eqref{eq:proj-K-reduced-2}.
If
 $\sum_{i=1}^n (\bar{x}_i)^2/(\bar{d}_i)_+\le k \bar{a}_+$,
then $(\bar{a}_+,\bar{x},\bar{d}_+)$ is feasible for \eqref{eq:proj-K-reduced-2}, and hence it is the projection onto the feasible set of \eqref{eq:proj-K-reduced-2}.
Otherwise, consider the dual problem of \eqref{eq:proj-K-reduced-2}:
\begin{equation} \label{eq:Lagrangian}
\max_{\lambda\ge 0}\min_{a\geq 0,d\geq 0,x}\left\{
L(a,x,d;\lambda):=\frac12(a-\bar a)^2+\|x-\bar x\|^2+\frac12\|d-\bar d\|^2
-\lambda k a + \lambda \sum_{i=1}^n \frac{x_i^2}{d_i}
\right\}.
\end{equation}
Strong duality holds by Slater's condition. Let $(a^*,x^*,d^*)$ be optimal for 
\eqref{eq:proj-K-reduced-2}
and $\lambda^*$ be optimal for \eqref{eq:Lagrangian} . If $\lambda^*=0$, then 
$(a^*,x^*,d^*)=(\bar{a}_+,\bar{x},\bar{d}_+)$, which contradicts the assumption that $\sum_{i=1}^n (\bar{x}_i)^2/(\bar{d}_i)_+ > k \bar{a}_+$. Hence $\lambda^*>0$.

Minimizing the Lagrangian $L(a,x,d;\lambda)$ with respect to $x$ first gives $x_i=d_i\bar x_i/(d_i+\lambda)$ for all $i\in[n]$. Thus the inner subproblem in \eqref{eq:Lagrangian} reduces to 
\begin{equation}
\label{eq:reduced}
\big(a(\lambda),d(\lambda)\big) := \argmin_{a\ge 0,d\ge 0}
\left\{
\frac12(a-\bar a)^2-\lambda ka
+\sum_{i=1}^n
\left[
\frac12(d_i-\bar d_i)^2+\frac{\lambda\bar x_i^2}{d_i+\lambda}
\right]
\right\}.
\end{equation}
Thus $a(\lambda)=(\bar a+k\lambda)_+$, and each $d_i(\lambda)$ is the unique minimizer of the stated one-dimensional problem \eqref{eq:d-lambda}. If $\bar d_i +{\bar x_i^2}/{\lambda}\le0$, then the function  
of $d_i$ within the square brackets in \eqref{eq:reduced} is nondecreasing on $[0,\infty)$, and hence $d_i(\lambda)=0$; otherwise, the first-order condition yields \eqref{eq:cubic-K}. Since $\lambda^*>0$, complementary slackness implies that the constraint $\sum_{i=1}^n x_i^2/d_i\le ka$ is active at the optimum. Substituting the formulas for $x_i^* = {d_i(\lambda^*)\bar{x}_i}/(d_i(\lambda^*) + \lambda^*)$ and 
$a^* = a(\lambda^*)$ into this equality gives \eqref{eq:root-K}. 

Any multiplier $\lambda^*>0$ satisfying \eqref{eq:root-K}, together with the formulas above, satisfies the KKT conditions and therefore yields the unique projection $(a^*,x^*,d^*)$.
\end{proof}

\begin{remark}
\label{rem:appendix-proj-K-newton}
To find an optimal multiplier $\lambda^*$, we solve the scalar equation \eqref{eq:root-K}. On the branch where $d_i(\lambda)>0$, implicit differentiation of \eqref{eq:cubic-K} gives
\[
\frac{{\rm d}\, d_i(\lambda)}{{\rm d}\lambda}
=
\frac{(d_i(\lambda)-\lambda)\bar x_i^2}{(d_i(\lambda)+\lambda)^3}
\bigg/
\left(1+\frac{2\lambda \bar x_i^2}{(d_i(\lambda)+\lambda)^3}\right).
\]
Hence \eqref{eq:root-K} can be handled efficiently by a Newton method.
\end{remark}

\begin{remark}
In case {\rm(ii)} of Proposition~\ref{prop:appendix-proj-K}, the multiplier $\lambda^*$ belongs to the interval $[0,U]$, where $U:=(-\bar a+\sqrt{\bar a^2+\|\bar x\|^2})/(2k)$.
Indeed, if
$\phi(\lambda):=\sum_{i=1}^n \bar x_i^2 d_i(\lambda)/(d_i(\lambda)+\lambda)^2-k(\bar a+k\lambda)_+$,
then the elementary bound $d/(d+\lambda)^2\le 1/(4\lambda)$ for $d\ge 0$ implies
$\phi(\lambda)\le \|\bar x\|^2/(4\lambda)-k(\bar a+k\lambda)_+$.
Since case {\rm(ii)} implies $\phi(0^+)>0$ and $\phi(\lambda^*)=0$, the definition of $U$ gives $\phi(U)\le 0$, and hence $\lambda^*\in[0,U]$. This bound is convenient for safeguarded Newton or bisection, although it is generally not tight.
\end{remark}

\begin{remark}
For problems with additional nonnegativity constraints $x\ge 0$, the RLT system \eqref{eq:intro-RLT} naturally induces the lifted constraint $Y\in\mathcal N$, where \(\mathcal N\) is defined in \eqref{eq:nonnegative-cone}. The projection onto $\K\cap\mathcal N$ is then obtained from the projection onto $\K$ with only two changes: the off-diagonal entries of the lower-right block are projected independently onto the nonnegative orthant, and the vector $\bar x$ in the structured subproblem is replaced by $(\bar x)_+$. Consequently, the same one-dimensional reduction as in Proposition~\ref{prop:appendix-proj-K} applies after this replacement.
\end{remark}

\subsection{Geometric properties of \texorpdfstring{$\K^*$}{K*}}\label{subsec:property-Kstar}

The dual cone $\K^*$ is defined by
\[\K^*=\{W\in\S^{n+1}:\langle W,Y\rangle\ge 0\ \forall Y\in \K\}.\]
The dual cone $\K^*$ plays a central role in both the multiplier updates of the augmented Lagrangian method and the dual-fixing analysis. We begin by deriving an explicit representation of $\K^*$. The following completion lemma provides the key ingredient.

\begin{lemma}
\label{lem:arrow-psd-completion}
Let $\alpha\in\R$, $z\in\R^n$, and $d\in\R^n$. Then the following are equivalent:
\begin{enumerate}[label=(\roman*), itemsep=0pt, topsep=2pt]
\item there exists $U\in\S^n$ with $\diag(U)=d$ such that
$
\begin{bmatrix}
\alpha & z^\top\\
z & U
\end{bmatrix}\succeq 0;
$
\item $\alpha\ge 0$, $d\ge 0$, and $z_i^2\le \alpha d_i$ for all $i\in[n]$.
\end{enumerate}
\end{lemma}

\begin{proof}\;
(i) $\Rightarrow$ (ii). If (i) holds, then positive semidefiniteness implies $\alpha\ge 0$, $d_i=U_{ii}\ge 0$, and every principal $2\times 2$ minor gives $z_i^2\le \alpha d_i$.

(ii) $\Rightarrow$ (i). Assume (ii). If $\alpha=0$, then necessarily $z=0$, and choosing $U=\Diag(d)$ gives (i). If $\alpha>0$, define $U:=\alpha^{-1}zz^\top+\Diag(d-\alpha^{-1}(z\circ z))$, where $z\circ z=(z_1^2,\dots,z_n^2)^\top$. Then $\diag(U)=d$, the diagonal correction is nonnegative by (ii), and $U-\alpha^{-1}zz^\top\succeq 0$. Hence the Schur complement implies (i).
\end{proof}

\begin{proposition}
\label{prop:Kstar-explicit}
\label{prop:Kstar-parameterization}
Every $W\in\K^*$ admits a unique representation of the form
\[
W=
\begin{bmatrix}
k\alpha & z^\top\\
z & \Diag(d)
\end{bmatrix},
\]
where
$
\alpha\ge 0,\; d\ge 0,\;
z_i^2\le \alpha d_i,\; i\in[n].
$
\end{proposition}

\begin{proof}\;
Since $\K=\{Y\in\S^{n+1}:\D(Y)\succeq 0\}$, standard conic duality gives
$\K^*=\D^*(\S^{n+1}_+)$. Moreover, for any $T=[\alpha, z^\top; z, U]\in \S^{n+1}$, we have
$\D^*(T)=[k\alpha, z^\top; z, \Diag(U)]$. Hence $W\in\K^*$ if and only if
$W=[k\alpha, z^\top; z, \Diag(d)]$
for some $\alpha\in\R$, $z\in\R^n$, and $d\in\R^n$ such that there exists
$U\in\S^n$ with $\diag(U)=d$ and $[\alpha, z^\top; z, U]\succeq 0$. By Lemma~\ref{lem:arrow-psd-completion}, this is equivalent to
$\alpha\ge 0$, $d\ge 0$, and $z_i^2\le \alpha d_i$ for all $i\in[n]$, which proves the displayed characterization. Uniqueness follows immediately from the entries of $W=[k\alpha, z^\top; z, \Diag(d)]$.
\end{proof}

The multiplier update associated with $\K^*$ can be recovered from $\Pi_\K$ through the Moreau decomposition,
\begin{equation}\label{rem:moreau}
    \bar Y=\Pi_\K(\bar Y)-\Pi_{\K^*}(-\bar Y),
\end{equation}
so the algorithm only needs the projection routine for $\K$.

\subsection{Projection efficiency}\label{subsec:projection-efficiency}

{
In this subsection, we evaluate the efficiency of the projection routine onto $\K$, which is one of the main operations used by our algorithm RiNNAL-S to be developed later in Section~\ref{sec:algorithm} for solving \eqref{eq:intro-Decomposed-SDP}. The nonnegative variant $\K\cap\mathcal{N}$, with \(\mathcal N\) defined in \eqref{eq:nonnegative-cone}, has the same one-dimensional structure after the modifications described in Subsection~\ref{subsec:projection-K}, and its behavior is similar. We therefore report the data generation and numerical results for $\K\cap\mathcal N$ in Appendix~\ref{app:additional-projection}.

For the test instances, we use the same generation protocol across the tested dimensions and report representative values of $n$. For $\K$, we first generate feasible matrices with $a,d_i\in[0.5,1.5]$ and $\sum_i x_i^2/d_i=0.7ka$, where the vector $x$ has random signs, and then add a symmetric perturbation.

We compare our projection scheme based on solving \eqref{eq:root-K} 
in Proposition~\ref{prop:appendix-proj-K}
with MOSEK 11.0 \cite{mosek}
for solving \eqref{eq:proj-K-reduced}.
In Table~\ref{tab:projection-benchmark-direct}, `Speedup' denotes the runtime ratio when both solvers finish, and `RelErr' denotes the relative difference $\|Y_{\mathrm{Ours}}-Y_{\mathrm{MOSEK}}\|_F/(1+\|Y_{\mathrm{MOSEK}}\|_F)$.
If MOSEK exceeds the time limit, the corresponding `Speedup' and `RelErr' entries are reported as ``--''.

\small
\begin{longtable}{lrrrrr}
\caption{Projection results on $\K$.}
\label{tab:projection-benchmark-direct}\\
\toprule
Set & $n$ & Ours (s) & MOSEK (s) & Speedup & RelErr \\
\midrule
\endfirsthead
\caption[]{continued from previous page}\\
\toprule
Set & $n$ & RiNNAL-S (s) & MOSEK (s) & Speedup & RelErr \\
\midrule
\endhead
\midrule
\multicolumn{6}{r}{Continued on next page}\\
\endfoot
\bottomrule
\endlastfoot
$\K$   & 50    & 0.0034 & 2.6      & 754      & 1.28e-06 \\
$\K$   & 200   & 0.0086 & 2448.2   & 283636   & 3.26e-07 \\
$\K$   & 1000  & 0.0520 & $>3600.0$  & --       & -- \\
$\K$   & 10000 & 2.9000 & $>3600.0$  & --       & -- \\
\end{longtable}
\normalsize

The results in Table~\ref{tab:projection-benchmark-direct} show that our projection scheme based on \eqref{eq:root-K} requires substantially less time than MOSEK for projection onto $\K$. The runtime gap widens with the dimension, reaching several orders of magnitude before MOSEK hits the time limit on large instances, while the relative errors remain uniformly small. Together with the supplementary results in Appendix~\ref{app:additional-projection}, this supports the implementation choice of using the projection $\Pi_\K$ as the basic 
operation within the RiNNAL-S algorithm to be developed
later for solving \eqref{eq:intro-Decomposed-SDP}.
}

\section{SC-SDP relaxation}\label{sec:sc-sdp}

In this section, we develop the SC-SDP relaxation and the associated dual objects used later for presolving and algorithm design. Subsection~\ref{subsec:relaxation-formulation} presents the relaxation formulation and its dual counterpart, Subsection~\ref{subsec:tightness-comparison} compares its tightness with the standard SDP and SDP--RLT relaxations. Subsection~\ref{subsec:applications} studies the unconstrained sparse ridge regression setting and its connection with the optimal perspective relaxation.

\subsection{Relaxation formulation}\label{subsec:relaxation-formulation}

\noindent\tb{SDP relaxation.}
The standard SDP relaxation for \eqref{eq:MIQP} is constructed by introducing auxiliary variables to represent all quadratic products of $x$ and $z$, relaxing the binary variables to continuous ones, and linearizing the nonconvex constraints through positive semidefinite constraints. The resulting SDP relaxation can be written as follows:
\begin{equation}\label{eq:Shor}
\min\left\{
\langle Q,X_{xx} \rangle+2c^\top x  \;\middle|\;
\begin{aligned}
&Ax=b,\;Bx\geq d,\; x = \diag(X_{xz}),\; z = \diag(X_{zz})\\
& 
\begin{bmatrix}
    1&x^\top & z^\top\\
    x&X_{xx} & X_{xz}\\
    z&X_{xz}^\top & X_{zz}
\end{bmatrix}\succeq 0,\; z\in \Z
\end{aligned}
\right\},
\end{equation}
where $\mathcal{Z}\coloneqq\{z\in [0,1]^n: e^\top z \le k\}$.
The SDP relaxation provides a valid lower bound for \eqref{eq:MIQP} but is not tight in general.

\medskip
\noindent\tb{SDP--RLT relaxation.}
To further strengthen \eqref{eq:Shor}, we add the RLT constraints in \eqref{eq:intro-RLT} derived from the original linear constraints. This yields the SDP--RLT relaxation
\begin{equation}\label{eq:Shor-RLT}
\min\left\{
\langle Q,X_{xx} \rangle+2c^\top x  \;\middle|\;
\begin{aligned}
&\A(Y)=0,\; \B(Y)\geq 0,\; x = \diag(X_{xz}),\; z = \diag(X_{zz})\\
& 
\begin{bmatrix}
    1&x^\top & z^\top\\
    x&X_{xx} & X_{xz}\\
    z&X_{xz}^\top & X_{zz}
\end{bmatrix}\succeq 0,\; z\in \Z
\end{aligned}
\right\}.
\end{equation}
where $\A(Y)$ and $\B(Y)$ are defined in \eqref{eq:intro-RLT}.
The SDP--RLT relaxation provides a significantly tighter bound than the standard SDP relaxation alone. However, it is computationally prohibitive for large-scale problems due to the large semidefinite matrix with dimension $2n+1$ and the large number of RLT constraints.

\medskip
\noindent\tb{Reduced SDP--RLT relaxation.}
To avoid the dimensional bottleneck of \eqref{eq:Shor-RLT}, we propose a reduced SDP--RLT relaxation that preserves the SDP--RLT bound while reducing the size of the semidefinite matrix. To construct it, we introduce an auxiliary variable $X=xx^\top$ to linearize the quadratic term in the objective function and represent the complementarity constraints in \eqref{eq:MIQP} by the linear matrix inequalities (LMIs)
\[
\begin{bmatrix}
    z_i&x_i\\x_i & X_{ii}
\end{bmatrix}\succeq 0,\qquad i\in [n].
\]
Then, by relaxing the rank-one constraint $X=xx^\top$, we obtain the lifted matrix
\begin{equation}\label{eq:Y}
    Y
\;\coloneq\; 
\begin{bmatrix}
    Y_{11} &Y_{12}\\Y_{21} & Y_{22}
\end{bmatrix}
\;=\;
\begin{bmatrix}
    1&x^\top \\
    x&X
\end{bmatrix}\succeq 0
\end{equation}
and, together with the RLT constraints, the reduced SDP--RLT relaxation of \eqref{eq:MIQP} becomes
\begin{equation}
\label{eq:Decomposed-SDP-z}
\min_{Y,z}\left\{
\langle \bQ,Y \rangle  \;\middle|\;  
\begin{aligned}
&\A(Y)=0,\; \B(Y)\geq 0,\;Y_{11}=1,\; Y\succeq 0\\
& \begin{bmatrix}
z_i & x_i\\
x_i & X_{ii}
\end{bmatrix}\succeq 0,\ i\in [n],\;z\in \mathcal{Z}
\end{aligned}
\right\},
\end{equation}
where the cost matrix $\bQ\coloneqq [ 0, c^\top; c, Q]$.
As we shall show in Theorem~\ref{thm:equivalence}, the reduced SDP--RLT relaxation \eqref{eq:Decomposed-SDP-z} is equivalent to the SDP--RLT relaxation \eqref{eq:Shor-RLT} and has a much smaller semidefinite matrix with dimension $n+1$ instead of $2n+1$.
However, \eqref{eq:Decomposed-SDP-z} still keeps the explicit support vector $z$ and $n$ separate LMIs, which are not conducive for algorithm design and presolving.

\medskip
\noindent\tb{SC-SDP relaxation.}
The SC-SDP relaxation is obtained by 
eliminating the variable $z$ and its associated constraint $z\in \mathcal{Z}$, and 
replacing the $n$ separate LMIs in \eqref{eq:Decomposed-SDP-z} by a single structured constraint that still captures the sparsity information. The explicit formulation \eqref{eq:intro-Decomposed-SDP} is provided in Subsection \ref{subsec:SDP-RLT-formulation}, where the cone $\K$ is defined by \eqref{eq:intro-K-set}.
The dual problem of \eqref{eq:intro-Decomposed-SDP} can be written as
\begingroup
\renewcommand{\theequation}{D}
\begin{equation}\label{eq:dual-Decomposed-SDP}
\max\left\{ \alpha
\;\middle|\;
\begin{aligned}
&\bQ-\A^* (\Gamma) - \B^* (\Lambda)-\D^*(W) - S-\alpha E_{11}=0,\\
&\alpha\in\R,\;\Gamma\in\S^{m},\;\Lambda\in\R^{(l+1)\times(l+1)}_+,\;W\in\S^{n+1}_+,\;S\in\S^{n+1}_+
\end{aligned}
\right\}, \tag{D}
\end{equation}
\endgroup
Here, $E_{11}$ denotes the matrix with 1 in the $(1,1)$-entry and 0 elsewhere. The KKT conditions for \eqref{eq:intro-Decomposed-SDP} and \eqref{eq:dual-Decomposed-SDP} can be written as
\begin{equation}
\label{eq:KKT-Decomposed-SDP}
\begin{aligned}
&\A(Y)=0,\; \B(Y)\geq 0,\; Y_{11}=1,\; Y\in \K\cap \S^{n+1}_+,\\
&\bQ-\A^* (\Gamma) - \B^* (\Lambda)-\D^*(W) - S-\alpha E_{11}=0,\\
&\langle \B(Y),\Lambda\rangle=0,\; \langle \D(Y),W\rangle=0,\; \langle Y,S\rangle=0,\; 
\Lambda\geq 0,\; W\succeq 0,\; S\succeq 0.
\end{aligned}
\end{equation}

The formulations above provide the basis for the remainder of this section. In particular, Subsection~\ref{subsec:tightness-comparison} shows that \eqref{eq:intro-Decomposed-SDP} has the same optimal value as the reduced SDP--RLT relaxation \eqref{eq:Decomposed-SDP-z} and the SDP--RLT relaxation \eqref{eq:Shor-RLT}.  Subsection~\ref{subsec:applications} relates SC-SDP to the optimal perspective relaxation in the special case of unconstrained sparse ridge regression setting. The dual problem 
\eqref{eq:dual-Decomposed-SDP} introduced here will be used again in Section~\ref{sec:presolving} to derive presolving rules and a refined dual certificate.

\subsection{Tightness comparison}\label{subsec:tightness-comparison}

The next result places the SC-SDP relaxation in the usual hierarchy of relaxations for \eqref{eq:MIQP}. In particular, it has the same optimal value as the reduced SDP--RLT and SDP--RLT relaxations, while remaining tighter than the standard SDP relaxation.

\begin{theorem}
\label{thm:equivalence} 
The following inequality chain holds:
\[
\emph{val} \eqref{eq:Shor}
\;\le\;
\emph{val} \eqref{eq:Shor-RLT}
\;=\;
\emph{val} \eqref{eq:Decomposed-SDP-z}
\;=\;
\emph{val} \eqref{eq:intro-Decomposed-SDP}
\;\le\; \emph{val} \eqref{eq:QP}.
\]
\end{theorem}

\begin{proof}\;
We first compare the SDP--RLT relaxation \eqref{eq:Shor-RLT} with the reduced SDP--RLT relaxation \eqref{eq:Decomposed-SDP-z}. Let a feasible solution of \eqref{eq:Shor-RLT} be given and set $Y=[1, x^\top; x, X_{xx}]$.
Since $Y$ is a principal submatrix of the lifted SDP matrix, $Y\succeq0$. Moreover, each principal submatrix indexed by $(x_i,z_i)$ gives
\[
\begin{bmatrix}
X_{xx,ii}&X_{xz,ii}\\
X_{xz,ii}&X_{zz,ii}
\end{bmatrix}
=
\begin{bmatrix}
X_{xx,ii}&x_i\\
x_i&z_i
\end{bmatrix}
\succeq0,
\]
where $X_{xx,ii}$, $X_{xz,ii}$, and $X_{zz,ii}$ are the $(i,i)$-entries of $X_{xx}$, $X_{xz}$, and $X_{zz}$, respectively.
After permuting the two rows and columns, this is exactly the perspective LMI in \eqref{eq:Decomposed-SDP-z}. The RLT constraints depend only on $Y$, and the objective values agree because $\langle Q,X_{xx}\rangle+2c^\top x=\langle \bQ,Y\rangle$.
Hence every feasible solution of \eqref{eq:Shor-RLT} projects to a feasible solution of \eqref{eq:Decomposed-SDP-z} with the same objective value.

Conversely, let $(Y,z)$ be feasible for \eqref{eq:Decomposed-SDP-z}, where $Y=[1,x^\top;x,X]\succeq0$. Take a square factorization $Y=PP^\top$ with $P\in\R^{(n+1)\times(n+1)}$, and write the rows of $P$ as $p_0,\ldots,p_n\in\R^{n+1}$. For each $i\in[n]$, consider
\[
G_i=
\begin{bmatrix}
1&x_i&z_i\\
x_i&X_{ii}&x_i\\
z_i&x_i&z_i
\end{bmatrix}.
\]
The diagonal entries and the principal $2\times2$ minors of $G_i$ are nonnegative because $Y\succeq0$, $z_i\in[0,1]$, and $[z_i,x_i;x_i,X_{ii}]\succeq0$. Its determinant is $\det(G_i)=(1-z_i)(z_iX_{ii}-x_i^2)\ge0$. Since a symmetric $3\times3$ matrix is positive semidefinite if and only if all of its principal minors are nonnegative, we have $G_i\succeq0$.
To construct the required new row explicitly, let $P_i$ be the two-row matrix with rows $p_0$ and $p_i$, and set
\[
S_i:=P_iP_i^\top=[1,x_i;x_i,X_{ii}],
\qquad
b_i:=(z_i,x_i)^\top .
\]
Since $G_i=[S_i,b_i;b_i^\top,z_i]\succeq0$, the generalized Schur complement gives $b_i\in\Range(S_i)$ and $z_i-b_i^\top S_i^\dagger b_i\ge0$. Define $u_i=P_i^\top S_i^\dagger b_i\in\R^{n+1}$. Then
\[
P_i u_i=S_iS_i^\dagger b_i=b_i,
\qquad
\|u_i\|^2=b_i^\top S_i^\dagger b_i .
\]
Let $\delta_i:=z_i-\|u_i\|^2\ge0$. Embed the factor rows in $\R^{2n+1}$ by replacing each $p_j$ with $(p_j,0)$, and continue to write $P$ for the matrix formed by these zero-padded rows. If $e_i$ is the $i$-th coordinate vector in $\R^n$, set
\[
q_i:=(u_i,\sqrt{\delta_i}\,e_i)\in\R^{2n+1}.
\]
Then $\langle p_0,q_i\rangle=z_i$, $\langle p_i,q_i\rangle=x_i$, and $\langle q_i,q_i\rangle=z_i$. The remaining inner products involving the vectors $q_i$ are not prescribed by \eqref{eq:Decomposed-SDP-z}; define $(X_{xz})_{ji}:=\langle p_j,q_i\rangle$ and $(X_{zz})_{ij}:=\langle q_i,q_j\rangle$.
Let $P_z\coloneq[q_1,\ldots,q_n]^\top$. Then
\[
\begin{bmatrix}
1&x^\top&z^\top\\
x&X&X_{xz}\\
z&X_{xz}^\top&X_{zz}
\end{bmatrix}
=
\begin{bmatrix}
P\\
P_z
\end{bmatrix}
\begin{bmatrix}
P\\
P_z
\end{bmatrix}^{\top}
\succeq0.
\]
Moreover,
\[
\diag(X_{xz})=x,\qquad \diag(X_{zz})=z.
\]
The RLT constraints are unchanged because they only involve $Y$, and the objective value is also unchanged. Hence \eqref{eq:Shor-RLT} and \eqref{eq:Decomposed-SDP-z} have the same optimal value. 

Next, we compare \eqref{eq:Decomposed-SDP-z} with the SC-SDP relaxation \eqref{eq:intro-Decomposed-SDP}. The two formulations have the same objective function and the same constraints involving $\A(Y)$, $\B(Y)$, $Y_{11}=1$, and $Y\succeq0$. Let $(Y,z)$ be feasible for \eqref{eq:Decomposed-SDP-z}. The LMI $[z_i,x_i;x_i,X_{ii}]\succeq0$ gives $z_i\ge0$ and $z_iX_{ii}\ge x_i^2$. Hence, with the convention used in Lemma~\ref{lemma:equivalence},
\[
\sum_{i=1}^n \frac{x_i^2}{X_{ii}}
\le
\sum_{i=1}^n z_i
=e^\top z
\le k .
\]
By Lemma~\ref{lemma:equivalence}, this is equivalent to $Y\in\K$. Thus $Y$ is feasible for \eqref{eq:intro-Decomposed-SDP}.

Conversely, let $Y$ be feasible for \eqref{eq:intro-Decomposed-SDP}. By Lemma~\ref{lemma:equivalence}, $\sum_{i=1}^n x_i^2/X_{ii}\le k$.
Define
\[
z_i:=
\begin{cases}
x_i^2/X_{ii}, & X_{ii}>0,\\
0, & X_{ii}=0.
\end{cases}
\]
Since $Y\succeq0$, we have $X_{ii}\ge x_i^2$, and hence $0\le z_i\le1$. The preceding inequality gives $e^\top z\le k$, so $z\in\mathcal Z$. By construction, $z_iX_{ii}=x_i^2$ for every $i$, and therefore
\[
\begin{bmatrix}
z_i&x_i\\
x_i&X_{ii}
\end{bmatrix}
\succeq0,\qquad i\in[n].
\]
Hence $(Y,z)$ is feasible for \eqref{eq:Decomposed-SDP-z}. Therefore $\emph{val} \eqref{eq:intro-Decomposed-SDP}=\emph{val} \eqref{eq:Decomposed-SDP-z}$.

Moreover, the SDP--RLT relaxation is obtained by adding RLT constraints to the SDP relaxation, so $\emph{val} \eqref{eq:Shor}\le \emph{val} \eqref{eq:Shor-RLT}$. Finally, since \eqref{eq:MIQP} is an exact reformulation of \eqref{eq:QP}, both the SDP and SDP--RLT problems are relaxations of the original problem, and hence $\emph{val} \eqref{eq:Shor-RLT}\le \emph{val} \eqref{eq:QP}$.
\end{proof}

\begin{remark}
    Han et al.~\cite{han2022equivalence} established an elegant related reduction and equivalence result for a perspective-type SDP relaxation. Their analysis focuses on convex quadratic perspective relaxations; here we allow a nonconvex quadratic objective, incorporate the RLT constraints, and link the reduced formulation to the sparsity-cone SC-SDP model used later for computation and presolving.
\end{remark}

\subsection{Equivalence to optimal perspective 
in the unconstrained sparse ridge regression setting}\label{subsec:applications}

In this subsection, we use unconstrained sparse ridge regression to illustrate the behavior of SC-SDP in a setting where the optimal perspective relaxation is well defined. We show that SC-SDP dominates the standard perspective relaxation and is equivalent to the optimal perspective relaxation.
The sparse ridge regression problem can be formulated as the following optimization problem:
\begin{equation}\label{eq:srr-original}
v^{\mathrm{P}}
\;\coloneqq\;
\min_{x}
\left\{
\|Ax-b\|^2+\gamma \|x\|^2
\;\middle|\;
\|x\|_0\le k
\right\},
\end{equation}
where $A\in\R^{m\times n}$, $b\in\R^m$, $\gamma>0$, and $k\in[n]$ are given.
The perspective relaxation of \eqref{eq:srr-original} is
\begin{equation}\label{eq:srr-persp}
    v^{\mathrm{PR}}
\;\coloneqq\;
\min_{x,z}\;
 \left\{
    \|Ax-b\|^2+\gamma \sum_{i=1}^n \frac{x_i^2}{z_i}
    \;\middle|\;  
    z\in \mathcal{Z}
\right\},
\end{equation}
where $\mathcal{Z}=\{z\in [0,1]^n \mid  e^\top z \le k\}$.
The SC-SDP relaxation of \eqref{eq:srr-original} in the form of \eqref{eq:intro-Decomposed-SDP} can be written as
\begin{equation}\label{eq:srr-SDP}
v^{\mathrm{SC}}\;\coloneqq\;\min_{Y}\;
\Big\{
\langle C_\gamma,Y\rangle 
\;\Big|\;
Y_{11}=1,\;
Y\in \K\cap \S_+^{n+1}
\Big\},
\end{equation}
where 
\[
\begin{aligned}
C_\gamma \;\coloneqq\; C+\begin{bmatrix}
0 & 0^\top\\
0 & \gamma I_n
\end{bmatrix},\;\;
C\coloneqq \begin{bmatrix}
    b^\top\\-A^\top
\end{bmatrix}\begin{bmatrix}
b & -A
\end{bmatrix}.
\end{aligned}
\]

We first compare \eqref{eq:srr-SDP} with the standard perspective relaxation \eqref{eq:srr-persp} in the next proposition. 

\begin{proposition}\label{thm:SDP-perspective}
The SC-SDP relaxation \eqref{eq:srr-SDP} is tighter than the perspective relaxation \eqref{eq:srr-persp}, i.e., $v^{\mathrm{SC}} \ge v^{\mathrm{PR}}$.
\end{proposition}

\begin{proof}\;
The perspective relaxation \eqref{eq:srr-persp} admits the lifted reformulation
\begin{align}
    v^{\mathrm{PR}}
\;=\;&\min_{x,z,t}\;
 \left\{
    \|Ax-b\|^2+\gamma e^\top t
    \;\middle|\;  
    \begin{bmatrix}
z_i & x_i\\
x_i & t_i
\end{bmatrix}\succeq 0,\; i\in [n],\;
    z\in \mathcal{Z}
\right\}\label{eq:srr-persp-SOCP}\\
=&\min_{Y,z,t}\;
\Big\{
\langle C,Y\rangle + \gamma e^\top t
\;\Big|\;
Y_{11}=1,\;
Y \succeq 0,\;
\begin{bmatrix}
z_i & x_i\\
x_i & t_i
\end{bmatrix}\succeq 0,\ i\in [n],\;
z\in \mathcal{Z}
\Big\}.\label{eq:srr-persp-SDP}
\end{align}
On the other hand, by the equivalence \eqref{eq:kDiag-xx-general}, 
the SC-SDP relaxation \eqref{eq:srr-SDP} can be reformulated as
\begin{equation}\label{eq:srr-SDP-reformulated}
v^{\mathrm{SC}}\;\coloneqq\;\min_{Y,z,t}\;
\Big\{
\langle C,Y\rangle + \gamma e^\top t
\;\Big|\;
Y_{11}=1,\;
Y\succeq 0,\;
\begin{bmatrix}
z_i & x_i\\
x_i & t_i
\end{bmatrix}\succeq 0,\ i\in [n],\;
z\in \mathcal{Z},\; \diag(X)=t
\Big\}.
\end{equation}
The feasible set of \eqref{eq:srr-SDP-reformulated} is contained in that of \eqref{eq:srr-persp-SDP}, since it imposes the additional constraint $\diag(X)=t$. Because the two problems have the same objective, it follows that $v^{\mathrm{SC}} \ge v^{\mathrm{PR}}$.

\end{proof}

The next theorem sharpens this comparison by identifying \eqref{eq:srr-SDP} with the strongest relaxation in the perspective family.

\begin{theorem}\label{thm:srr-opt-persp}
The SC-SDP relaxation \eqref{eq:srr-SDP} is equivalent to the following optimal perspective relaxation:
\begin{equation}\label{eq:srr-opt-persp}
\max_{\mu\in\R^n_+}\;\min_{x,z}
\left\{
x^\top Q_\mu x + 2c^\top x+\sum_{i=1}^n \mu_i\;\frac{x_i^2}{z_i}
+b^\top b
\;\middle|\;
z\in\Z,\;
 Q_\mu \succeq 0
\right\},
\end{equation}
where $Q_\mu\coloneqq A^\top A+\gamma I-\Diag(\mu)$ and $c\coloneqq -A^\top b$.
\end{theorem}

\begin{proof}\;
Starting from \eqref{eq:srr-SDP-reformulated}, eliminate $t$ to obtain
\begin{equation}\label{eq:srr-SDP-reformulated-3}
\min_{Y,z}\;
\Big\{
\langle C,Y\rangle + \gamma e^\top \diag(X)
\;\Big|\;
Y_{11}=1,\;
Y\succeq 0,\;
z\in \mathcal{Z},\; X_{ii}\geq \frac{x_i^2}{z_i},\; i\in [n]
\Big\}.
\end{equation}
Problem \eqref{eq:srr-SDP-reformulated-3} admits a strictly feasible point, so strong duality holds. Its dual is exactly the optimal perspective relaxation \eqref{eq:srr-opt-persp} with the constant \(b^\top b\) retained.
\end{proof}

\section{Presolving with SC-SDP}\label{sec:presolving}

For sparsity-constrained problems such as \eqref{eq:QP} and \eqref{eq:MIQP}, the main combinatorial difficulty is to identify the support of the solution. A strong relaxation is therefore useful not only as a source of lower bounds, but also as a source of certificates that fix variables or exclude impossible support patterns before extensive branching. This section shows how the SC-SDP dual can be used for this purpose.

\medskip
\noindent\tb{Restricted relaxations.}
A direct way to test a support decision is to impose the corresponding restriction and solve the restricted relaxation. Although SC-SDP eliminates the explicit support vector $z$, the same sparsity-cone construction and algorithmic framework extend naturally to these restricted subproblems; see Appendix~\ref{app:support-restrictions}. In sparse ridge regression, the tightness results from Subsection~\ref{subsec:applications} imply that the resulting SC-SDP support-restriction bounds dominate the standard perspective bounds and match the optimal perspective bounds. Thus this route provides strong node-restriction bounds, with strong branching as a special case.

\medskip
\noindent\tb{Root-dual certificates.}
The difficulty is that support-restriction bounds require solving one relaxation for each candidate restriction, and the number of candidates can grow quickly when screening larger support patterns. Our alternative is to extract presolving certificates from a single optimal dual solution of the root SC-SDP relaxation. This is the main contribution of this section: Subsection~\ref{subsec:variable-fixing} derives dual-fixing rules for individual variables, Subsection~\ref{subsec:screening-cuts} derives screening cuts for larger support patterns, and Subsection~\ref{subsec:dual-refinement} introduces a dual-refinement step that strengthens these certificates without changing the SC-SDP lower bound.

\begin{remark}\label{rmk:node-certificates}

The same certificate construction can also be applied at a child node of a branch-and-bound tree. In this case, the node restrictions are incorporated into the SC-SDP relaxation, and the resulting node-level dual solution yields fixing and screening certificates with respect to the corresponding node bounds.

\end{remark}

\subsection{Variable fixing}\label{subsec:variable-fixing}
In this subsection, we show how to derive variable-fixing rules from an optimal dual solution of the SC-SDP relaxation \eqref{eq:intro-Decomposed-SDP}. These rules can then be used to reduce the search space of the original problem \eqref{eq:MIQP}.
The details are provided in the next theorem. 

\begin{theorem}[Dual fixing]\label{thm:dual-fixing}
Let $v^{\mathrm{UB}}$ be any upper bound on the optimal value $v^{\mathrm{P}}$ of~\eqref{eq:MIQP}.
Let $(\alpha, \Gamma, \Lambda, W, S)$ be feasible dual variables of~\eqref{eq:dual-Decomposed-SDP}
with objective value $v^{\mathrm{LB}}$. 
Then, for any $p\in[n]$, any optimal solution of~\eqref{eq:MIQP} $(x^*,z^*)$ satisfies
\[
z_p^*=\begin{cases}
    0,&\text{if}
    \quad w_{[k]}-w_p> v^{\mathrm{UB}}-v^{\mathrm{LB}},\\
    1,& \text{if}
    \quad
    w_p-w_{[k+1]}> v^{\mathrm{UB}}-v^{\mathrm{LB}},
\end{cases}
\]
where $w_i := {(W_{12})_i^2}/{(W_{22})_{ii}}$ if $(W_{22})_{ii}>0$, and $w_i:=0$ otherwise.
\end{theorem}

\begin{proof}\;
Let $(x,z)$ be any optimal solution of \eqref{eq:MIQP}. Define $Y=[1, x^\top; x, xx^\top] $ as the corresponding rank-one lifted matrix.  
By the dual feasibility of $(\alpha, \Gamma, \Lambda, W, S)$ and primal feasibility of $Y$, we have
\begin{align*}
v^{\mathrm{UB}} 
&\ge v^{\mathrm{P}} 
= x^\top Q x + 2c^\top x 
= \langle \bQ, Y \rangle 
= \alpha Y_{11}+\langle\A^*(\Gamma)+\B^*(\Lambda)+\D^*(W)+S,Y\rangle\\
&\ge v^{\mathrm{LB}}+\langle W,\D(Y)\rangle
=v^{\mathrm{LB}}+kW_{11}
+\sum_{i=1}^n\left((W_{22})_{ii}x_i^2+2(W_{12})_i x_i\right).
\end{align*} 

	 We prove the first case by contradiction.
	 Fix any $p\in[n]$.
Suppose $z_p=1$ and $w_{[k]}-w_p> v^{\mathrm{UB}}-v^{\mathrm{LB}}\geq 0$. 
Let \(J:=\{i\in[n]\mid z_i=1\}\). Then \(|J|\le k\), \(p\in J\), and \(i\notin J\) implies \(x_i=0\). Moreover, for every \(i\),
\[
(W_{22})_{ii}x_i^2+2(W_{12})_i x_i\ge -w_i,
\]
where the case \((W_{22})_{ii}=0\) follows from \(W\succeq0\), which implies \((W_{12})_i=0\).
Therefore, we have
\begin{align*}
v^{\mathrm{UB}} - v^{\mathrm{LB}}
&\geq k W_{11}+\sum_{i=1}^n\left((W_{22})_{ii}x_i^2+2(W_{12})_i x_i\right)
\geq k W_{11}-\sum_{i\in J}w_i\\
&\geq k W_{11}-\sum_{i=1}^{k} w_{[i]}+(w_{[k]}-w_p)
\geq w_{[k]}-w_p>v^{\mathrm{UB}}-v^{\mathrm{LB}},
\end{align*}
where the third inequality uses \(|J|\le k\) and \(p\in J\), and the fourth inequality uses \(W_{11}\geq w_i\) for all \(i\). This leads to a contradiction, and hence $z_p=0$. The second case is proved similarly by contradiction.
\end{proof}

\subsection{Screening cuts}\label{subsec:screening-cuts}

When dual fixing alone does not identify any variable, we can still generate valid cuts from the same dual solution. To state these screening cuts, we first recall the notion of an SCG-tuple from \cite{tan2025screening}.

\begin{definition}[\bf SCG-tuple]\label{def:scg-tuple}
Given two disjoint index subsets $S, N \subseteq [n]$ with $|S| \le k$, let
$
R := [n] \setminus (S \cup N)
$
denote the remaining index set. For any index set $C \subseteq R$, we say that
$(S,N,C)$ is an \emph{SCG-tuple} with respect to $w$ if
the index set $C$ is constructed as
\[
C := \{ i \in R \mid w_i \ge (w_R)_{[c]} \},
\]
where
$
c := \min\{k - |S|,\ |R|\}.
$
In particular, when $c = 0$, we set $C := \emptyset$.
Throughout this paper, for an SCG-tuple $(S,N,C)$, we refer to $S$ as the
\emph{candidate support set}, $N$ as the \emph{candidate non-support set}, and
$C$ as the \emph{complement support set} relative to $S$, which is determined
based on $S$, $N$, and $w$.
\end{definition}

Using the SCG-tuple, we can obtain the following screening cut from the dual solution of \eqref{eq:intro-Decomposed-SDP}.

\begin{theorem}[Screening cuts]\label{thm:screening-cuts}
Under the same assumptions as in Theorem~\ref{thm:dual-fixing}, let $(S,N,C)$ be an SCG-tuple with respect to $w$. If
\[
\sum_{i=1}^k w_{[i]} - \sum_{i\in S} w_i - \sum_{i\in C} w_i \;>\; v^{\mathrm{UB}} - v^{\mathrm{LB}},
\]
then the following screening cut is valid for all optimal solutions of~\eqref{eq:MIQP}:
\[
\sum_{i\in S} z_i + \sum_{i\in N} (1 - z_i) \;\le\; |S| + |N| - 1 .
\]
\end{theorem}

\begin{proof}\;
Let $(x,z)$ be an optimal solution of \eqref{eq:MIQP}, set $J:=\{i\in[n]\mid z_i=1\}$, and define $Y=[1,x^\top;x,xx^\top]$. Since $z$ is binary and $e^\top z\le k$, we have $|J|\le k$. Also, $\langle \bQ,Y\rangle=v^{\mathrm{P}}\le v^{\mathrm{UB}}$. Since $Y$ is feasible for \eqref{eq:intro-Decomposed-SDP}, and since $v^{\mathrm{LB}}=\alpha$, the dual stationarity equation gives $\langle\bQ,Y\rangle-v^{\mathrm{LB}}\ge\langle W,\D(Y)\rangle$, where
\[
\langle W,\D(Y)\rangle
=kW_{11}+\sum_{i=1}^n (W_{22})_{ii}x_i^2+2\sum_{i=1}^n (W_{12})_i x_i .
\]
Here the inequality uses $\A(Y)=0$, $\B(Y)\ge0$, $\Lambda\ge0$, and $Y,S\succeq0$.
For each $i$, the principal $2\times2$ submatrix of $W$ indexed by $(1,i+1)$ is positive semidefinite. Hence $(W_{22})_{ii}\ge0$ and $W_{11}\ge w_i$. Moreover,
\[
(W_{22})_{ii}x_i^2+2(W_{12})_i x_i\ge -w_i .
\]
Indeed, this follows by completing the square when $(W_{22})_{ii}>0$, while if $(W_{22})_{ii}=0$, positive semidefiniteness implies $(W_{12})_i=0$ and both sides are zero. If $i\notin J$, then $z_i=0$ and the complementarity constraint in \eqref{eq:MIQP} gives $x_i=0$, so the corresponding term is zero. Therefore
\[
v^{\mathrm{UB}}-v^{\mathrm{LB}}
\ge
kW_{11}-\sum_{i\in J}w_i
\ge
\sum_{i=1}^k w_{[i]}-\sum_{i\in J}w_i .
\]

Suppose, to obtain a contradiction, that this optimal solution violates the proposed cut. Since each term in the left-hand side of the cut is at most one, violation is possible only if $z_i=1$ for all $i\in S$ and $z_i=0$ for all $i\in N$. Thus $S\subseteq J$ and $J\cap N=\emptyset$. Let $R=[n]\setminus(S\cup N)$ and $L:=J\backslash S$. Then $J=S\cup L$ and $|L|\le k-|S|$. By the definition of an SCG-tuple, $C$ contains the $c=\min\{k-|S|,|R|\}$ largest scores in $R$. Since $L\subseteq R$, $|L|\le c$, and $w_i\ge0$, we have
\[
\sum_{i\in L}w_i\le \sum_{i\in C}w_i .
\]
Consequently, $\sum_{i\in J}w_i=\sum_{i\in S}w_i+\sum_{i\in L}w_i\le \sum_{i\in S}w_i+\sum_{i\in C}w_i$, and hence
\[
v^{\mathrm{UB}}-v^{\mathrm{LB}}
\ge
\sum_{i=1}^k w_{[i]}-\sum_{i\in S}w_i-\sum_{i\in C}w_i
>
v^{\mathrm{UB}}-v^{\mathrm{LB}},
\]
which is impossible. Therefore no optimal solution can violate the cut, so the cut is valid for all optimal solutions of \eqref{eq:MIQP}.
\end{proof}

\subsection{Dual refinement}\label{subsec:dual-refinement}

The presolving rules in Theorems~\ref{thm:dual-fixing} and \ref{thm:screening-cuts} both depend on the semidefinite multiplier $W\succeq0$ in the dual problem \eqref{eq:dual-Decomposed-SDP} through the scores
\[
w_i=\frac{W_{1,i+1}^2}{W_{i+1,i+1}},\qquad i\in[n],
\]
with the zero-denominator convention used in Theorem~\ref{thm:dual-fixing}. Since the dual optimal solution is generally not unique, different optimal certificates may give the same lower bound but different presolving strength. We therefore introduce a dual-refinement step to select a certificate that is more useful for presolving.

Fix an optimal triple $(\bar\alpha,\bar\Gamma,\bar\Lambda)$ and define $\bar C:=\bQ-\A^*(\bar\Gamma)-\B^*(\bar\Lambda)-\bar\alpha E_{11}$. A natural refinement is to select the remaining dual certificate by maximizing the score separation

\begin{equation}
\label{eq:dual-refinement-ideal}
\max\left\{
\sum_{i=1}^k w_{[i]}-\sum_{i=k+1}^n w_{[i]}
\;\middle|\;
\bar C-\D^*(W)\succeq0,\ W\succeq0
\right\}.
\end{equation}
This problem is difficult to solve directly because the scores are fractional and the leading index set is unknown. To derive a tractable form, write
\[
\D^*(W)=
\begin{bmatrix}
k\tau & z^\top\\
z & \Diag(d)
\end{bmatrix}\in\K^*,
\]
where $\tau=W_{11}$, $z_i=W_{1,i+1}$, and $d_i=W_{i+1,i+1}$. Then $w_i=z_i^2/d_i\le\tau$. A leading score with $0<w_i<\tau$ can be saturated by reducing $d_i$ to $z_i^2/\tau$, which only decreases the corresponding diagonal entry of the arrow matrix and therefore preserves feasibility. Hence, after this saturation step, the leading scores are taken at the common upper envelope $\tau$, and we remove the unknown ordering as follows: 
\begin{equation*}
\sum_{i=1}^k w_{[i]}-\sum_{i=k+1}^n w_{[i]}
=2\sum_{i=1}^k w_{[i]}-\sum_{i=1}^n w_i
=2k\tau-\sum_{i=1}^n w_i .
\end{equation*}
Introducing variables $t_i$ with $w_i\le t_i\le\tau$, and we refine the dual certificate by solving
\begin{equation}
\label{eq:dual-refinement-uqp}
\min\left\{-2k\tau + \sum_{i=1}^n t_i
\;\middle|\;
\begin{aligned}
&\bar C-
\begin{bmatrix}
k\tau & z^\top\\
z & \Diag(d)
\end{bmatrix}
\succeq 0,\\
&\ d\ge 0,\ z_i^2\le t_i d_i,\ \tau \geq t_i \ge 0,\ i\in[n]
\end{aligned}
\right\}.
\end{equation}
\begingroup
\makeatletter
\def\@currentlabel{\theequation}
\label{eq:dual-refinement-uqp-reduced}
\makeatother
\endgroup

The refinement keeps the SC-SDP bound unchanged because $(\bar\alpha,\bar\Gamma,\bar\Lambda)$ is fixed and only the remaining dual certificate is reselected. Compared with the ordered formulation \eqref{eq:dual-refinement-ideal}, \eqref{eq:dual-refinement-uqp-reduced} is tractable: the ordering is removed, and the constraints \(z_i^2\le t_i d_i\) are rotated second-order cone representable. We use \eqref{eq:dual-refinement-uqp-reduced} later in the numerical experiments when constructing fixing certificates from an optimal dual solution.

\section{Algorithm framework for solving SC-SDP}\label{sec:algorithm}

This section develops the algorithmic framework for solving the SC-SDP relaxation \eqref{eq:intro-Decomposed-SDP}. In Subsection~\ref{subsec:alm}, we introduce the augmented Lagrangian framework. In Subsections~\ref{subsec:low-rank} and \ref{subsec:convex-lifting}, we describe the low-rank and convex lifting phases for solving the ALM subproblems. The projection onto $\K$ and the Moreau-based recovery of the $\K^*$ multiplier update are given in Section~\ref{sec:cone-geometry}; in Subsection~\ref{subsec:projection-F}, we derive the remaining projection onto $\F$.

\subsection{Augmented Lagrangian method}\label{subsec:alm}
We first equivalently express \eqref{eq:intro-Decomposed-SDP} in the following splitting form:
\begin{equation}
\label{eq:Decomposed-SDP-2}
\min_{Y,Z}\left\{
\langle \bQ,Y \rangle + \delta_{\F}(Y)+\delta_{\K}(Z)  \;\middle|\;  
 \B(Y)\geq 0,\; Y-Z=0
\right\},
\end{equation}
where $\delta_{\cal S}$ denotes the indicator function of the set ${\cal S}$, and the set $\F$ is defined as
\begin{equation}\label{eq:F-set}
    \F\coloneq\left\{Y\in\S^{n+1} : \A(Y)=0,\; Y_{11}=1,\; Y\in \S^{n+1}_+\right\}.
\end{equation}
The reduced augmented Lagrangian function of \eqref{eq:Decomposed-SDP-2} is
\begin{align}
\label{eq:reduced-AL-Decomposed-SDP}
L_\sigma(Y,\Lambda,W)
=& \langle \bQ,Y \rangle  + \frac{\sigma}{2}
\|(\sigma^{-1}\Lambda- \B(Y) )_+\|^2 + \frac{\sigma}{2}
\|\Pi_{\K^*}(\sigma^{-1}W- Y)\|^2 + \delta_{\F}(Y),
\end{align}
where $\Lambda\geq 0$ and $W\in \K^*$ are the Lagrange multipliers associated with the constraints $\B(Y)\geq 0$ and $Y-Z=0$, respectively, and $\sigma>0$ is the penalty parameter.
The augmented Lagrangian method for solving \eqref{eq:Decomposed-SDP-2} is summarized in Algorithm 1.

\begin{description}[itemsep=0pt]
\item[Algorithm 1:] Choose $Y_0\in \F$, $\Lambda_0\geq 0$, $W_0\in \K^*$, and $\sigma_0>0$. Set $k=0$.
\item[Step 1:] Compute $Y_{k+1}=\arg\min\{L_{\sigma_k}(Y,\Lambda_k,W_k): Y\in \F\}$. \hfill (ALM-sub)
\item[Step 2:] Update $\Lambda_{k+1}=(\Lambda_k-\sigma_k\B(Y_{k+1}))_+$.
\item[Step 3:] Update $W_{k+1}=\Pi_{\K^*}(\sigma_k^{-1}W_k-Y_{k+1})$.
\item[Step 4:] Choose $\sigma_{k+1}$, set $k\leftarrow k+1$, and return to Step 1.
\end{description}

\begin{remark}
Algorithm~1 follows a double-loop augmented Lagrangian framework. Under the standard assumptions used in inexact ALM analysis, if each ALM subproblem is solved to progressively higher accuracy so that the 
accumulated inexactness error is summable,
then the outer ALM iterates generate a primal-dual sequence whose accumulation points satisfy the KKT conditions of \eqref{eq:intro-Decomposed-SDP}; see \cite{RiNNAL} for a detailed analysis of this classical inexact ALM setting.

At the subproblem level, each ALM subproblem is solved by a hybrid strategy consisting of a low-rank phase and a convex lifting phase. Although the low-rank phase is nonconvex, the convex lifting step guarantees a monotonic decrease of the augmented Lagrangian value. As shown in \cite{lee2022escaping,RiNNALPOP}, this mechanism ensures global convergence of the subproblem iterates to a minimizer of (ALM-sub), while automatically identifying the true optimal rank.
\end{remark}

The next two subsections briefly describe how to implement the low-rank and convex lifting phases for solving the subproblem (ALM-sub) efficiently.

\subsection{Low-rank phase}\label{subsec:low-rank}

In the low-rank phase, we represent a rank-$r$ solution of (ALM-sub) in the factorized form
\[
Y=\begin{bmatrix} 1& x^\top\\ x & X\end{bmatrix}=\begin{bmatrix}
    e_1^\top\\R 
\end{bmatrix}\begin{bmatrix}
    e_1 & R^\top
\end{bmatrix} = \widehat{R}\widehat{R}^\top,
\]
where $R\coloneqq[R_1,\dots,R_n]^\top\in\R^{n\times r}$ and $\widehat{R}\coloneqq [e_1^\top;R]$ with $e_1$ being the first standard unit vector in $\R^r$. Thus, (ALM-sub) is equivalent to the following nonconvex problem:
\begin{equation}\label{eq:LR-Decomposed-complex}
\min_{R}\left\{
    f_r(R)\coloneq L_{\sigma}(\widehat{R}\widehat{R}^\top,\Lambda,W)
\;\middle|\;
\widehat R\widehat{R}^\top\in \F
\right\}.
\end{equation}
Directly solving \eqref{eq:LR-Decomposed-complex} is difficult because of the nonconvex constraint. However, \cite{RiNNAL,RiNNALplus} show that it can be reduced to
\begin{equation}\label{eq:LR-Decomposed}
\min_{R}\left\{
    f_r(R)
\;\middle|\;
AR=be_1^\top
\right\},
\end{equation}
which is a simpler nonconvex problem with only linear constraints. We then apply a projected gradient method to solve \eqref{eq:LR-Decomposed}.

\subsection{Convex lifting phase}\label{subsec:convex-lifting}

The low-rank phase serves as a local acceleration step for reducing the objective value of (ALM-sub), but it does not by itself guarantee global convergence of Algorithm 1. To ensure global convergence, we switch to a convex lifting phase after a suitable number of low-rank iterations. In this phase, we solve (ALM-sub) directly in the variable $Y$  via the projected gradient step
\begin{equation}\label{eq:PG-Decomposed}
Y^+=\Pi_{\F}\left(Y-\eta \nabla L_\sigma(Y,\Lambda,W)\right),
\end{equation}
where $\eta_k>0$ is a suitable step size. As shown in \cite{lee2022escaping,RiNNALPOP}, this phase guarantees global convergence of the proposed method and helps identify the optimal rank, which can in turn be used to warm start the low-rank phase
\eqref{eq:LR-Decomposed-complex}. The main computational bottleneck in \eqref{eq:PG-Decomposed} is the projection onto $\F$. The next subsection shows that this projection can be computed efficiently through a one-dimensional convex optimization problem.

\subsection{Projection onto \texorpdfstring{$\F$}{F}}\label{subsec:projection-F}
 
The projection onto $\K$ was developed in Section~\ref{sec:cone-geometry}, together with the Moreau-based recovery of the $\K^*$ multiplier update. We next derive the remaining projection formula for the affine-semidefinite set $\F$, which is needed in the convex lifting phase.
Recall that the set $\F$ defined in \eqref{eq:F-set} can be equivalently expressed as
\begin{align*}
   \F 
   =  \{Y\in\S^{n+1}: Y_{11}=1\}\cap\S^{\A}_{+},
\end{align*}
where $\S^{\A}_{+}\coloneqq\{Y\in\S^{n+1}_+:\A(Y)=0\}$.
Let $P =[b^\top;-A^\top]$ and $J=I-P(P^\top P)^{-1}P^\top$.
By \cite[Lemma 2]{RiNNALplus}, we have
\begin{equation}
    \label{eq:proj-SAplus}
    \Pi_{\S^{\A}_{+}}(\bar Y)=\Pi_{\S^{n+1}_+}(J\bar YJ), \quad \forall\, \bar Y\in\S^{n+1}.
\end{equation}
Therefore, projecting onto $\F$ reduces to a one-dimensional convex optimization problem.
\begin{theorem}
    \label{thm:projection-F}
    The projection of any given matrix $\bar Y\in \S^{n+1}$ onto $\F$ is given by
    \[
\Pi_{\F}(\bar Y)=\Pi_{\S^{\A}_{+}}(\bar Y+\eta^* E_{11}),
    \]
where $\eta^*$ is the unique optimal solution of the following problem:
    \begin{equation}
        \label{eq:proj-F}
        \min_{\eta\in\R}\;\left\{\frac{1}{2}\|\Pi_{\S^{n+1}_+}(J(\bar Y+ \eta E_{11})J)\|^2-\eta\right\}.
    \end{equation}
\end{theorem}

\begin{proof}\;
    The formula follows from the dual characterization of the projection onto $\F$ together with \eqref{eq:proj-SAplus}.
\end{proof}

The problem \eqref{eq:proj-F} is one-dimensional and convex, so it can be solved efficiently by standard methods such as the semismooth Newton method; this dual viewpoint is consistent with the affine-constrained proximal framework in \cite{hou2026efficient}. The computation can be further accelerated by preprocessing and warm-start strategies from \cite{RiNNALplus}, using the low-rank solution as an initial point.

\section{Numerical experiments}\label{sec:numerical}
In this section, we report numerical experiments to evaluate the efficiency and scalability of our algorithm RiNNAL-S for solving the SC-SDP relaxation. We also evaluate
the bound quality of the SC-SDP relaxation, its computational and presolving performance, and the broader applicability of the framework.
All experiments are performed using {\sc Matlab} R2023b
on a workstation equipped with Intel Xeon E5-2680 (v3) processors
and 96GB of RAM.

\medskip
\noindent\textbf{Baseline Solvers}. For the SC-SDP relaxation, we compare RiNNAL-S with the commercial solver MOSEK 11.0 \cite{mosek}. For the SDP--RLT relaxation, we report results obtained by RiNNAL+ \cite{RiNNALplus} and SDPNAL+ \cite{SDPNALp1,SDPNALp2}. For the original problem \eqref{eq:MIQP}, we report results from the commercial solver Gurobi 13.0.0 \cite{gurobi}, which solves the problem directly by branch-and-bound methods. These benchmarks are used to assess both the practical performance of SC-SDP and the computational advantage of RiNNAL-S.

\medskip
\noindent\textbf{Stopping Conditions}. Based on the KKT conditions \eqref{eq:KKT-Decomposed-SDP} for the SC-SDP relaxation \eqref{eq:Decomposed-SDP-2}, we define the following relative KKT residuals to assess the accuracy of the solution:
\begin{equation*}
\operatorname{R_p}:=
\max\left\{
\frac{{ \|\Pi_{-}(\B(Y))\|}}{1+\|d\|^2},\;
\frac{\|Y-Z\|}{1+\|Y\|+\|Z\|}
\right\},
\;\;
\operatorname{R_d} :=  
\frac{\|\Pi_{\S_-^{n+1}}(S)\|}{1+\|S\|},
\;\;
\operatorname{R_{c}} := 
\frac{|\langle Y, S\rangle|}{1+\|Y\|+\|S\|}.
\end{equation*}
\noindent
We omit residuals associated with the feasibility constraints
$Y\in \F$, $Z\in \K$, the dual feasibility constraints $\Lambda\geq 0$, $W\in \K^*$, and the complementarity constraints,
as the first is automatically enforced by the projection step in the ALM subproblem, and the others are inherently satisfied by the construction of the ALM iteration.
For a given tolerance $\tol> 0$, our algorithm terminates when the maximum residual satisfies $\operatorname{R_{max}}:=\max\{\operatorname{R_p},\operatorname{R_d},\operatorname{R_c}\}<\tol$ or the maximum time limit $\timelimit$ is reached. In our experiments, unless otherwise specified, we set $\tol = 10^{-6}$ and $\timelimit = 3600\mathrm{(secs)}$ for all solvers.

\medskip
\noindent\textbf{Implementation}. 
In RiNNAL-S, we employ a projected gradient (PG) method with Barzilai-Borwein steps and non-monotone line search to solve the augmented Lagrangian subproblems. Although classical ALM theory often increases the penalty parameter monotonically to obtain fast local convergence, this rule can be overly conservative and may lead to difficult subproblems. We therefore use an adaptive update in the implementation: the penalty parameter $\sigma_k$ is initialized at $\sigma_0 = 1$, increased by a factor of 1.5 if $\operatorname{R_p}/\operatorname{R_d}\geq 2$ and decreased by a factor of 1.5 if $\operatorname{R_p}/\operatorname{R_d}\leq 1/5$. The initial rank $r_0$ is set to be $\min\{200,\lceil n/5 \rceil \}$. The stepsize $t_k$ of the PG step is determined as $1/\sigma_k$. The initial point $R_0$ is randomly selected from the feasible region. In each ALM iteration, the low-rank phase is performed with the maximum number of PG iterations capped at 50. Additionally, a PG step is executed after the low-rank phase at every 5 outer ALM iterations.

\medskip
\noindent\textbf{Table Notations.} We use `-' to indicate that an algorithm did not achieve the required accuracy $\tol$ within the maximum time limit $\timelimit$. For MOSEK, the iteration count is not reported. 
A superscript ``$\dagger$'' is appended to the KKT residual value to indicate that the algorithm attained moderate accuracy, but not the required accuracy.
For relaxation comparisons, the relative gap is defined as $\mathrm{relgap}=(UB-LB)/\max\{1,|UB|\}$, and is multiplied by $100$ when reported in percentage. The Gurobi relative gap refers instead to the solver-reported MIP gap.
For the column labeled ``Iteration'' associated with SDPNAL+, the first entry denotes the number of outer iterations, the second entry denotes the total number of semismooth Newton inner iterations, and the third indicates the total number of ADMM+ iterations. Similarly, for the column labeled ``Iteration'' associated with RiNNAL-S and RiNNAL+, the first entry corresponds to the number of ALM iterations, the second denotes the total number of projected gradient descent 
iterations, and the third 
reports the total number of PG steps.
The column labeled ``Rank'' indicates the number of columns of the final iterate $R$ for RiNNAL-S and RiNNAL+ and the final rank of the output matrix $Y$ for SDPNAL+. 
The ``Objective'' column denotes the value of the objective function, while the total computation time is listed under ``Time''.

\subsection{Sparse ridge regression}

In this subsection, we investigate problem \eqref{eq:srr-scaling-model} as a representative problem class for evaluating the practical performance of SC-SDP. It is given by
\begingroup
\renewcommand{\theequation}{SRR}
\begin{equation}\label{eq:srr-scaling-model}
\min_{x\in\R^n}\left\{\frac{1}{m}\|Ax-b\|_2^2+\gamma\|x\|_2^2 \;\middle|\; \|x\|_0\le k\right\},
\tag{SRR}
\end{equation}
\endgroup
where $A\in\R^{m\times n}$, $b\in\R^m$, and $\gamma>0$ are given. We focus on two aspects: the scalability of RiNNAL-S with respect to the problem dimension, and its presolving performance.
For the synthetic instances used in this subsection, we follow the data generation schemes in \cite{tan2025screening,bertsimas2020sparse}. In particular, each row of $A$ is sampled independently from $\mathcal{N}(0,\Sigma)$ with Toeplitz covariance matrix $\Sigma_{ij}=\rho^{|i-j|}$. The ground-truth vector $x^*\in\{-1,0,1\}^n$ has exactly $k$ nonzero entries, with uniformly random support and independent random signs, and the response vector is generated by $b=Ax^*+\varepsilon$, where $\varepsilon\sim\mathcal{N}(0,\sigma^2 I)$ and $\sigma^2=\|Ax^*\|_2^2/(m\,\mathrm{SNR}^2)$. Unless otherwise specified, we use $m=2n$, $k=5$, $\rho=0.1$, $\gamma=1$ and $\mathrm{SNR}=1$.

\subsubsection{Scalability comparison}
\label{subsec:scalability}

We assess the scalability of RiNNAL-S with respect to the problem dimension. We compare RiNNAL-S and MOSEK on the SC-SDP relaxation \eqref{eq:intro-Decomposed-SDP}, together with Gurobi on the original problem \eqref{eq:srr-scaling-model}. For these instances, the SC-SDP relaxation is exact, so solving it with RiNNAL-S also solves \eqref{eq:srr-scaling-model}.

\begin{figure}[ht!]
    \centering
    \includegraphics[width=0.5\textwidth]{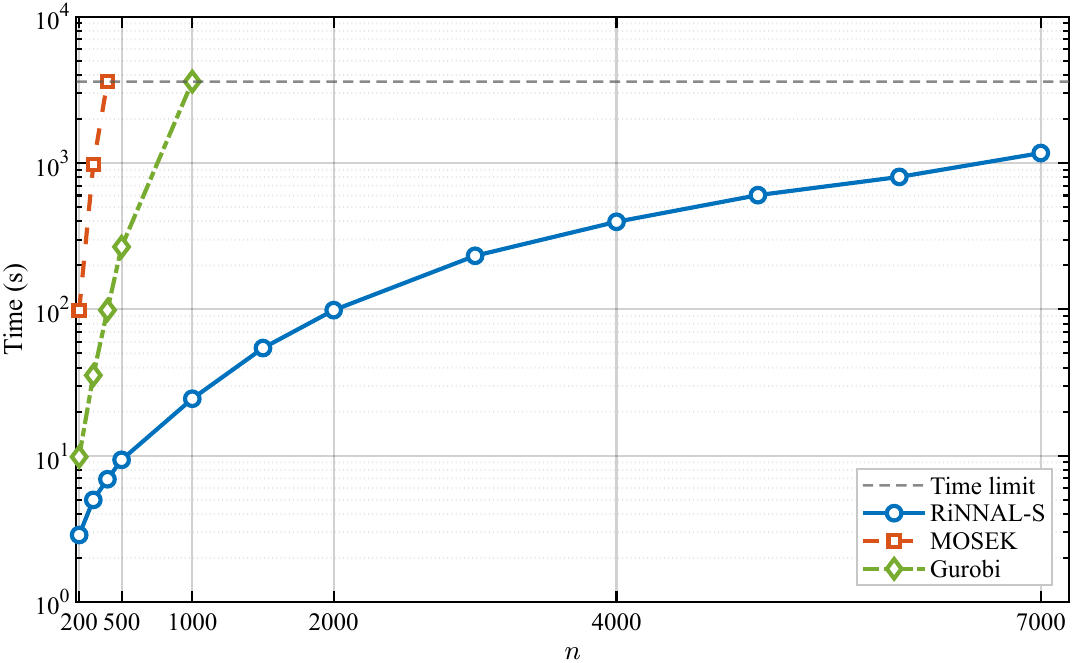}
    \caption{Scalability for \eqref{eq:srr-scaling-model}.}
    \label{fig:srr-scaling}
\end{figure}

Figure~\ref{fig:srr-scaling} shows that RiNNAL-S requires less time than both MOSEK and Gurobi on these instances. At $n=1000$, RiNNAL-S takes about $25$ seconds, whereas MOSEK requires at least one hour and still does not solve the instance; Gurobi also reaches the time limit at this dimension. The trend of the curves indicates that this computational gap continues to widen as the dimension increases, and RiNNAL-S remains effective up to $n=7000$. This experiment demonstrates the practical scalability of RiNNAL-S for solving the SC-SDP relaxation, and highlights the computational advantage of the proposed method over state-of-the-art commercial solvers on large instances.

\subsubsection{Presolving comparison}\label{subsubsec:srr-presolving-comparison}

We evaluate the presolving effect of SC-SDP through variable fixing, screening cuts, and their effect on solving the original problem with Gurobi. We also study the additional benefit of the dual-refinement step from Subsection~\ref{subsec:dual-refinement}, using the perspective relaxation \eqref{eq:srr-persp} as a baseline. The dual-refinement (DR) certificate derived from the optimal dual solution of SC-SDP is computed by MOSEK from the reduced formulation \eqref{eq:dual-refinement-uqp-reduced}. We use the same synthetic setting as above, but set $n=500$ and $\gamma=0.1$ to obtain a regime in which the presolving behavior of the perspective relaxation and SC-SDP can be clearly separated.
Note that under the parameters in Subsection \ref{subsec:scalability}
with $\gamma = 1$, both relaxations are already tight enough to fix all variables. 
All plots are based on five random instances and $\mathrm{SNR}\in\{0.40,0.42,\ldots,0.50\}$, with markers showing the mean and shaded bands showing the range from minimum to maximum. We first compare relaxation gaps and variable-fixing strength, then evaluate screening cuts, and finally test their practical effect in Gurobi.

\begin{figure}[ht!]
    \centering
    \begin{subfigure}[t]{0.32\textwidth}
        \centering
        \includegraphics[width=\textwidth]{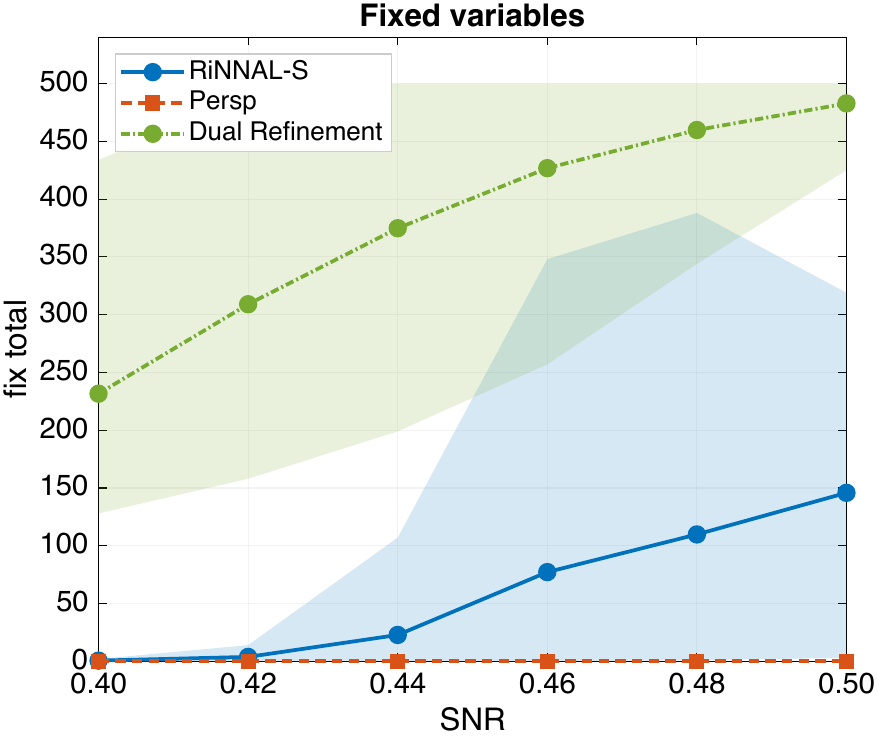}
        \caption{Fixed variables.}
        \label{fig:srr-fixing-snr-fix}
    \end{subfigure}
    \hfill
    \begin{subfigure}[t]{0.32\textwidth}
        \centering
        \includegraphics[width=\textwidth]{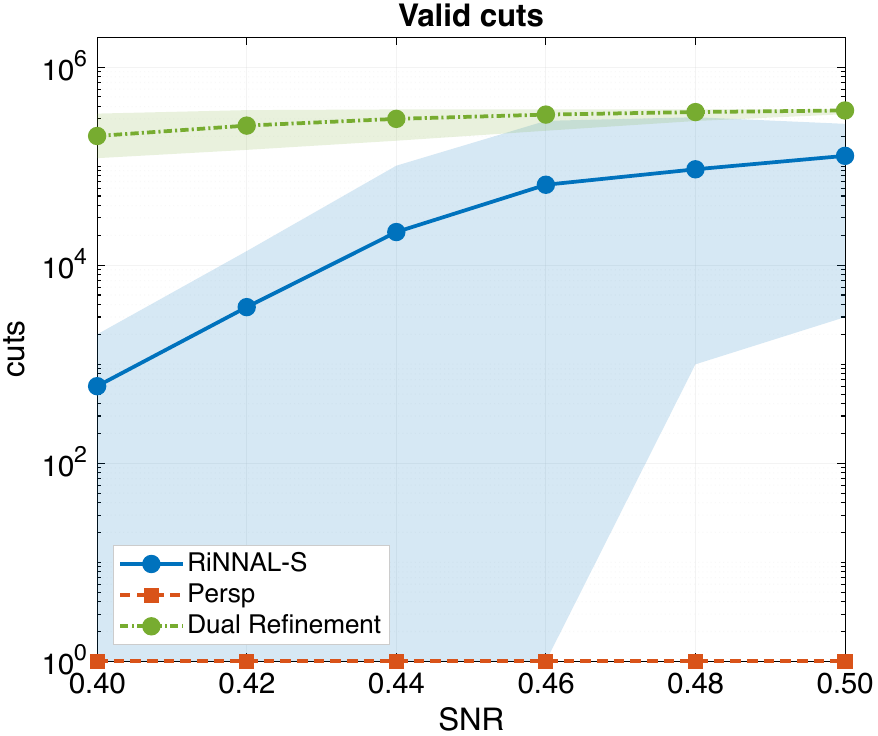}
        \caption{Valid cuts.}
        \label{fig:srr-fixing-snr-valid-cuts}
    \end{subfigure}
    \hfill
    \begin{subfigure}[t]{0.32\textwidth}
        \centering
        \includegraphics[width=\textwidth]{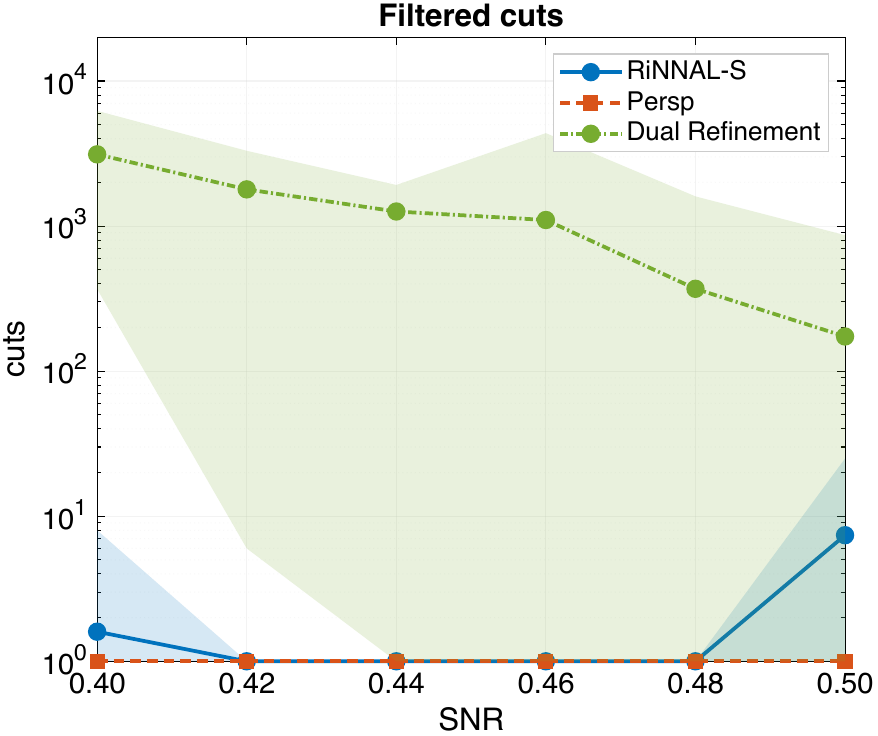}
        \caption{Filtered cuts.}
        \label{fig:srr-fixing-snr-filtered-cuts}
    \end{subfigure}\\[0.8em]
    \begin{subfigure}[t]{0.32\textwidth}
        \centering
        \includegraphics[width=\textwidth]{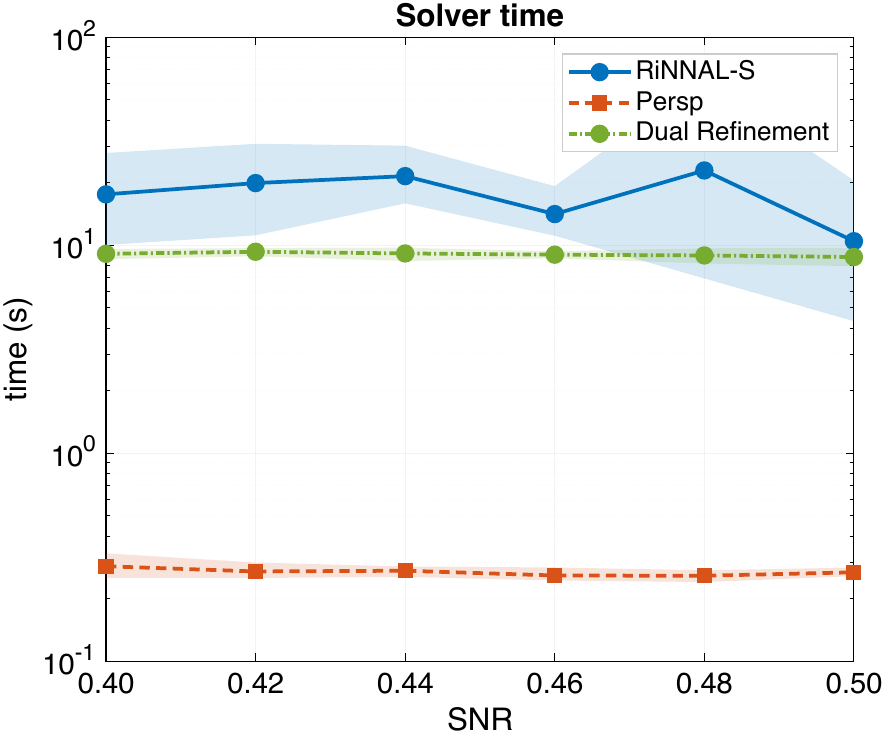}
        \caption{Solver time.}
        \label{fig:srr-fixing-snr-time}
    \end{subfigure}
    \hfill
    \begin{subfigure}[t]{0.32\textwidth}
        \centering
        \includegraphics[width=\textwidth]{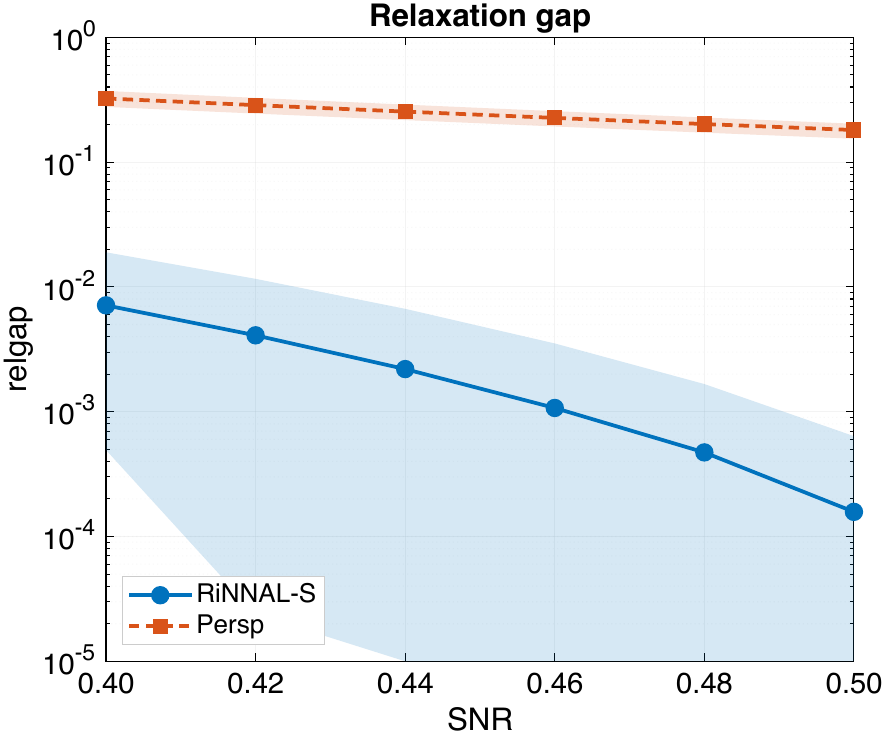}
        \caption{Relaxation gap.}
        \label{fig:srr-fixing-snr-relgap}
    \end{subfigure}
    \hfill
    \begin{subfigure}[t]{0.32\textwidth}
        \centering
         \includegraphics[width=\textwidth]{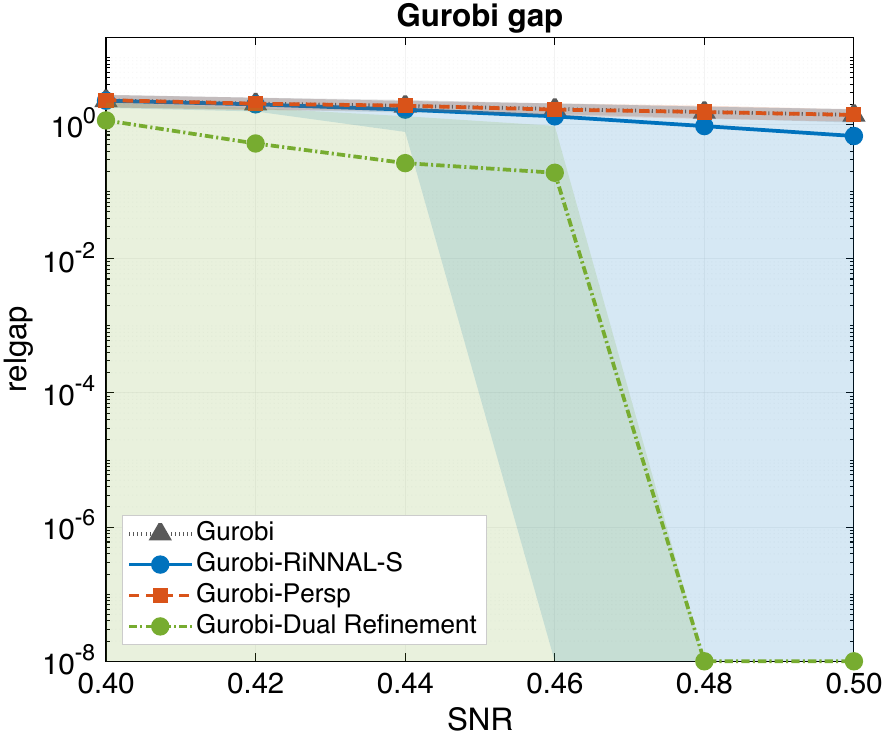}
        \caption{Gurobi gap (600s).}
        \label{fig:srr-fixing-snr-gurobi-gap}
    \end{subfigure}
    \caption{Presolving and Gurobi performance for \eqref{eq:srr-scaling-model} under varying SNR.}
    \label{fig:srr-fixing-snr}
\end{figure}

\medskip
\noindent\textbf{Relaxation gap.}
For each relaxation, the relative gap is computed using an upper bound obtained from its own support estimate: we first round the relaxation solution of either \eqref{eq:intro-Decomposed-SDP} 
or \eqref{eq:Decomposed-SDP-z}
to obtain a support set, then solve the restricted sparse ridge subproblem on that support. Since the support is fixed, this subproblem is a strongly convex unconstrained quadratic problem, and the upper bound is therefore obtained by solving the corresponding linear system.
Although the perspective relaxation is faster, Figure~\ref{fig:srr-fixing-snr}(\subref{fig:srr-fixing-snr-relgap}) shows that RiNNAL-S produces much smaller relaxation gaps throughout the tested SNR range. Thus, SC-SDP provides a much stronger lower bound even when the perspective relaxation is still too weak to identify the support.

\medskip
\noindent\textbf{Fixing.}
We apply the dual-fixing rule from Theorem~\ref{thm:dual-fixing} to the SC-SDP certificate and its dual-refined version from Subsection~\ref{subsec:dual-refinement}, and compare the resulting fixing performance with that of the perspective relaxation proposed in \cite{atamturk2020safe}.
Figure~\ref{fig:srr-fixing-snr}(\subref{fig:srr-fixing-snr-fix}) shows that SC-SDP fixes many more variables than the perspective relaxation, and the advantage becomes more visible as the 
relaxation \eqref{eq:intro-Decomposed-SDP} becomes tighter (i.e., smaller relaxation gap)
with increasing SNR. Dual refinement strengthens this effect further while keeping the SC-SDP bound unchanged, and the extra cost 
incurred by MOSEK in solving 
\eqref{eq:dual-refinement-uqp-reduced} is modest.

\medskip
\noindent\textbf{Screening.}
Here we apply the screening rule from Theorem~\ref{thm:screening-cuts} and consider pairwise screening cuts, corresponding to SCG tuples with $|S|+|N|=2$. We call a cut filtered if it remains after removing those already implied by variable fixing. Figures~\ref{fig:srr-fixing-snr}(\subref{fig:srr-fixing-snr-valid-cuts}) and \ref{fig:srr-fixing-snr}(
\subref{fig:srr-fixing-snr-filtered-cuts}) show that screening provides additional presolving information beyond single-variable fixing. RiNNAL-S generates many valid screening cuts, whereas the perspective relaxation yields very few, and dual refinement further increases this number substantially. The filtered-cut plot shows that a nontrivial portion remains after removing cuts already implied by variable fixing, so the refined dual certificate captures support-pattern information beyond fixing alone. For these two log-scale plots, zero cuts are displayed as one.

\medskip
\noindent\textbf{Presolving in Gurobi.}
We add the fixed variables and up to 100 filtered screening cuts to Gurobi and impose a 600-second time limit. Figure~\ref{fig:srr-fixing-snr}(\subref{fig:srr-fixing-snr-gurobi-gap}) shows that the 600-second MIP relative gap improves over solving the original formulation directly. 
We can see that the quality of variable fixing and filtered screening cuts provided by SC-SDP with dual refinement is substantially better than that provided by perspective relaxation \eqref{eq:srr-persp}, which provides no useful information and hence does not
improve Gurobi's MIP gap over the baseline.
This is consistent with the role of fixing and screening in presolving: they shrink the feasible region before branching and thereby reduce the branch-and-bound burden.

\begin{remark}
    The Gurobi relative gap may be much larger than the gaps of the perspective and SC-SDP relaxations, because the MIP gap is computed after adding fixing and screening information rather than using the relaxation lower bound itself. A more specialized branch-and-bound method that incorporates the SC-SDP bound more directly could further improve practical performance, and we leave this direction for future work.
\end{remark}

\subsubsection{Extension to constrained sparse ridge regression}

We next consider a constrained variant of sparse ridge regression in order to illustrate that the presolving framework also applies when additional linear equalities are present. The problem takes the form
\begingroup
\renewcommand{\theequation}{SRR-E}
\begin{equation}\label{eq:srr-equality-model}
\min_{x\in\R^n}\left\{\frac{1}{m}\|Ax-b\|_2^2+\gamma\|x\|_2^2
\;\middle|\;
e^\top x=\beta,\; \|x\|_0\le k
\right\},
\tag{SRR-E}
\end{equation}
\endgroup
where the scalar $\beta=e^\top x^*$ matches the ground-truth vector. 
The data generation and presolving pipeline are the same as those used in the
presolving comparison for \eqref{eq:srr-scaling-model}. The main difference
from the unconstrained setting is that the perspective relaxation used in the
previous comparison is no longer applicable in this form once the additional
coupling constraint \(e^\top x=\beta\) is imposed, so it is not included as a
baseline here. Accordingly, in this subsection we focus only on SC-SDP and its
dual-refined certificate.

\begin{figure}[H]
    \centering
    \begin{subfigure}[t]{0.32\textwidth}
        \centering
        \includegraphics[width=\textwidth]{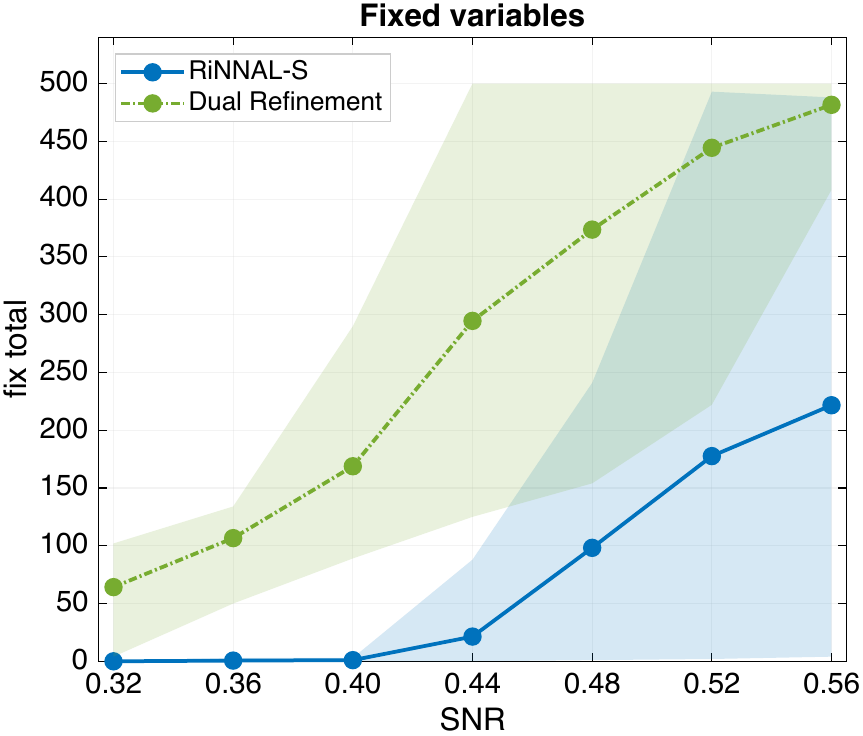}
        \caption{Fixed variables.}
        \label{fig:srr-equality-fixing-snr-fix}
    \end{subfigure}
    \hfill
    \begin{subfigure}[t]{0.32\textwidth}
        \centering
        \includegraphics[width=\textwidth]{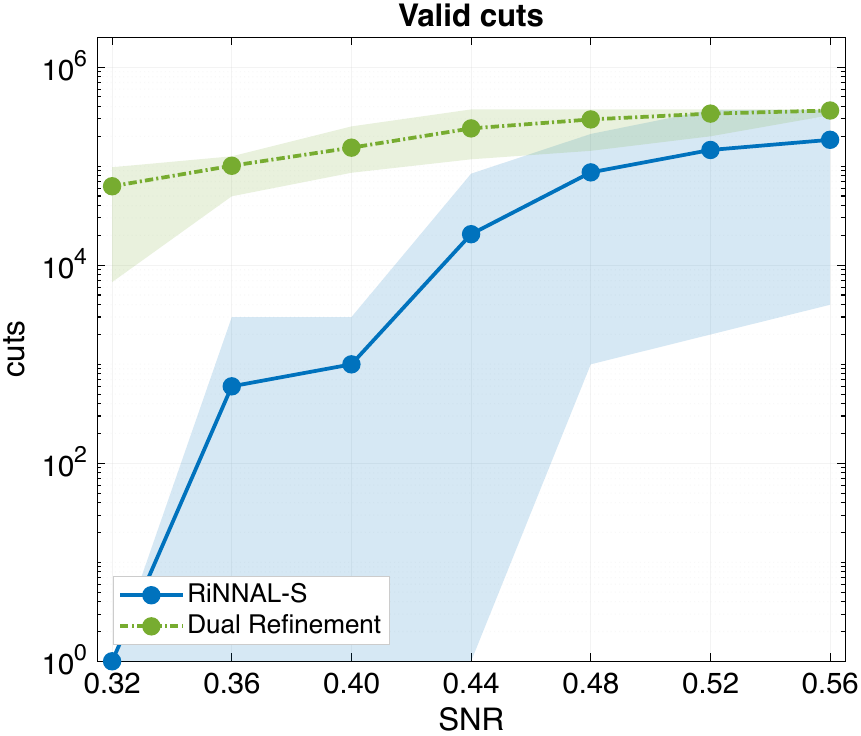}
        \caption{Valid cuts.}
        \label{fig:srr-equality-fixing-snr-valid-cuts}
    \end{subfigure}
    \hfill
    \begin{subfigure}[t]{0.32\textwidth}
        \centering
        \includegraphics[width=\textwidth]{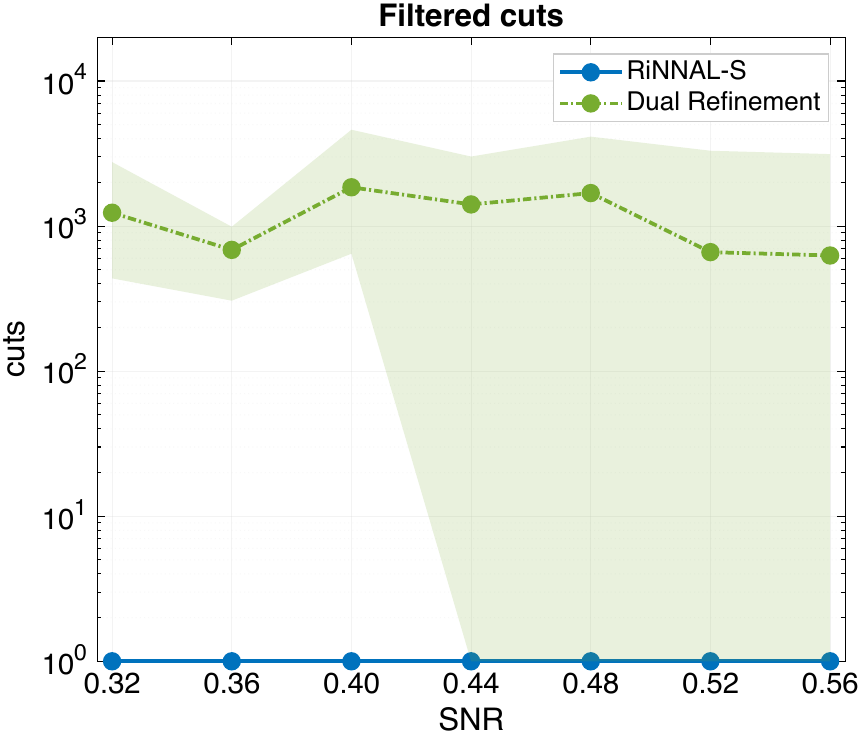}
        \caption{Filtered cuts.}
        \label{fig:srr-equality-fixing-snr-filtered-cuts}
    \end{subfigure}\\[0.8em]
    \begin{subfigure}[t]{0.32\textwidth}
        \centering
        \includegraphics[width=\textwidth]{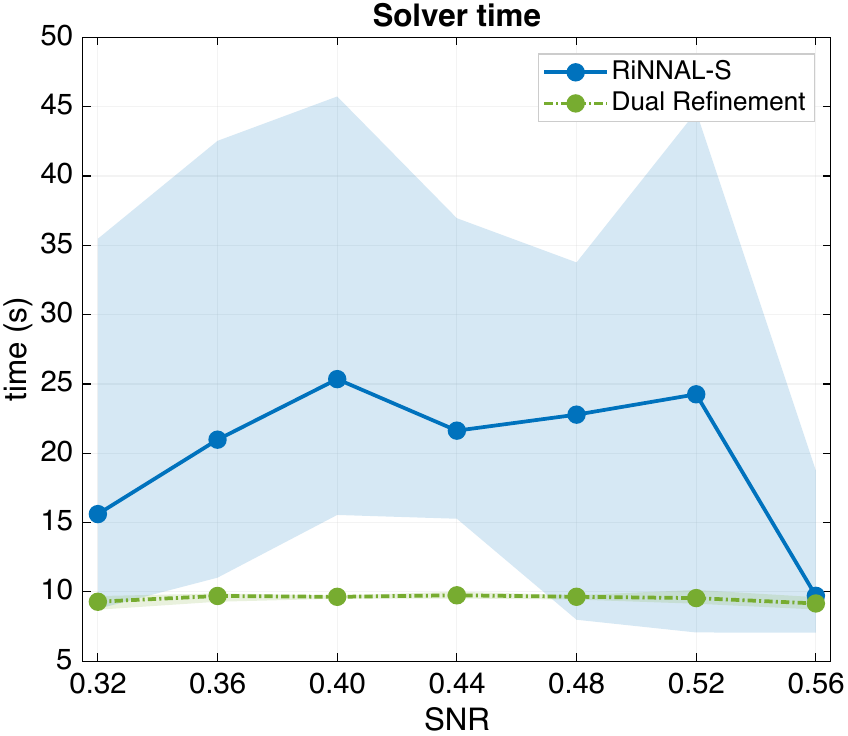}
        \caption{Solver time.}
        \label{fig:srr-equality-fixing-snr-time}
    \end{subfigure}
    \hfill
    \begin{subfigure}[t]{0.32\textwidth}
        \centering
        \includegraphics[width=\textwidth]{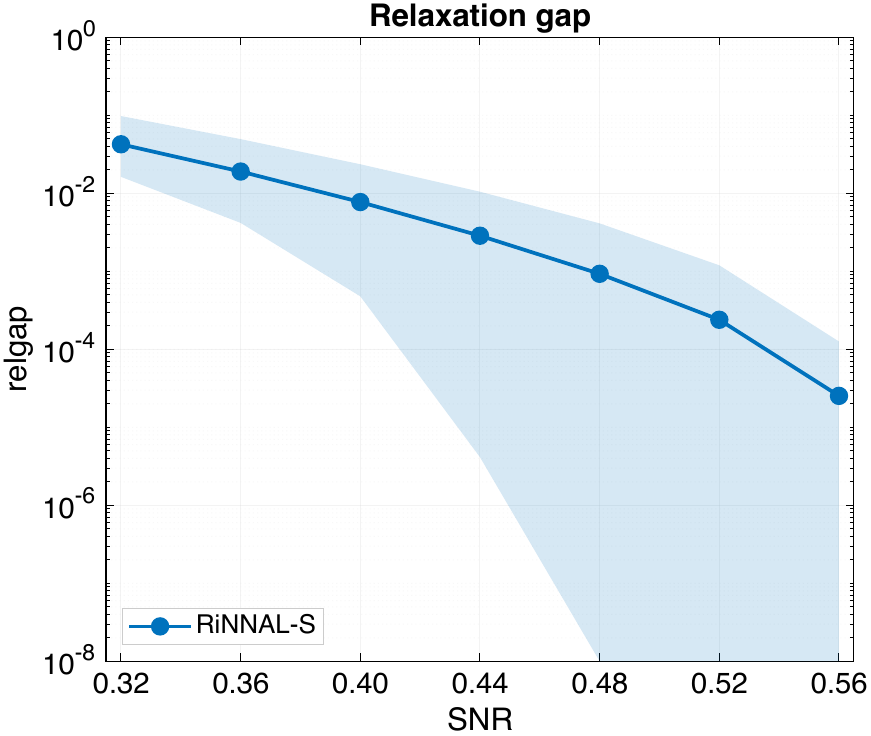}
        \caption{Relaxation gap.}
        \label{fig:srr-equality-fixing-snr-relgap}
    \end{subfigure}
    \hfill
    \begin{subfigure}[t]{0.32\textwidth}
        \centering
        \includegraphics[width=\textwidth]{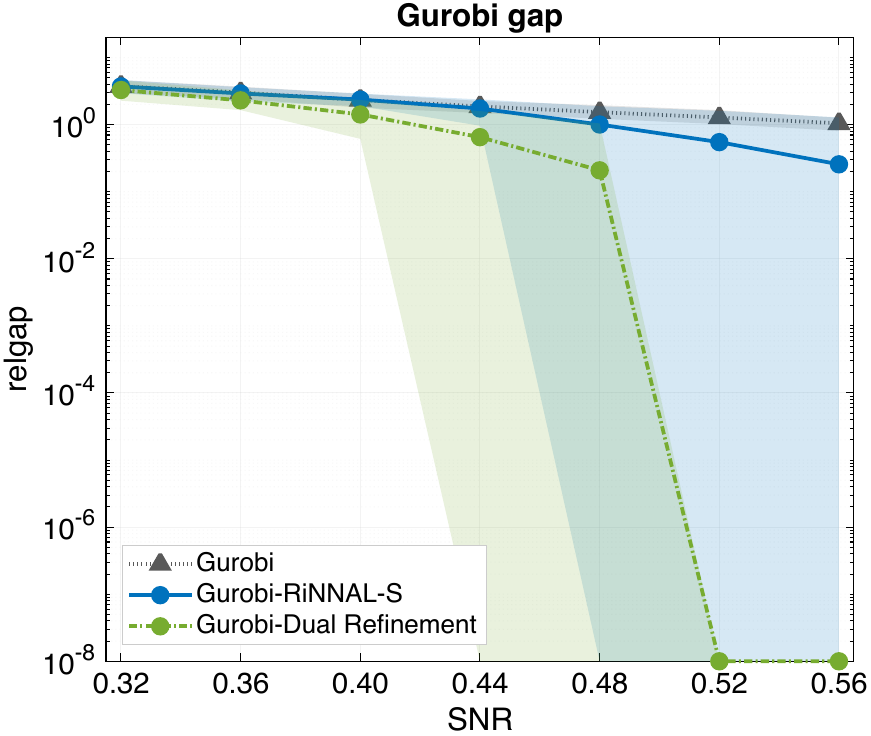}
        \caption{Gurobi gap (600s).}
        \label{fig:srr-equality-fixing-snr-gurobi-gap}
    \end{subfigure}
    \caption{Presolving and Gurobi performance for the equality-constrained sparse ridge regression model \eqref{eq:srr-equality-model} under varying SNR.}
    \label{fig:srr-equality-fixing-snr}
\end{figure}

Figure~\ref{fig:srr-equality-fixing-snr} confirms that SC-SDP presolving extends
to equality-constrained sparse ridge regression. Even when the relaxation gap is
not nearly zero, the dual-refined certificate still fixes many variables and
screens many support patterns, often leaving a much smaller original problem.
The smaller number of filtered cuts primarily reflects extensive fixing, which
leaves few remaining nontrivial cuts. The results by Gurobi 
in Figure~\ref{fig:srr-equality-fixing-snr}(f) further show that this
presolving information improves the downstream solver. These results highlight
SC-SDP as an effective presolver for more general constrained models.

\subsection{Sparse standard quadratic programs}\label{subsec:numexp-stqp-compare}

In this subsection, we study problem \eqref{eq:stqp-benchmark-model} as a benchmark class for evaluating the tightness and computational performance of SC-SDP. It takes the form
\begingroup
\renewcommand{\theequation}{S-StQP}
\begin{equation}\label{eq:stqp-benchmark-model}
\min_{x\in\R^n}
\left\{
x^\top Q x
\;\middle|\;
e^\top x = 1,\;
x\ge 0,\;
\|x\|_0\le k
\right\}.\tag{S-StQP}
\end{equation}
\endgroup
The matrix $Q$ is generated following the construction in \cite{bomze2025tighter,RiNNALplus}. More precisely, three classes of instances are considered, denoted by COP, PSD, and SPN. In terms of the observed solver difficulty for the resulting relaxations, COP instances are relatively easy, while PSD and SPN instances are harder. For each class, the data are generated so that the standard quadratic program without the sparsity constraint admits a unique global minimizer whose support size equals a target sparsity level of $k_0=n/4$. To evaluate the relaxations under a stricter sparsity requirement, the actual cardinality budget is set to $k=k_0/2$. {We examine three aspects: the bound quality of SC-SDP, its presolving performance, and the efficiency of RiNNAL-S on solving the resulting SC-SDP models.}

\subsubsection{Tightness comparison}\label{subsubsec:stqp-tightness-comparison}

We assess the bound quality delivered by SC-SDP against that of the standard SDP relaxation. For this comparison, the standard SDP relaxation \eqref{eq:Shor} and the SC-SDP relaxation \eqref{eq:intro-Decomposed-SDP} are both solved by MOSEK. The upper bound is obtained by selecting the relaxation solution with the best lower bound among the two methods, rounding it to a support set of cardinality $k$, and then solving the resulting fixed-support quadratic problem by Gurobi. Table~\ref{tab:stqp-relgap-values} reports the relative gap in percentage.

\small
\begin{longtable}{llrrrr}
\caption{Relative gaps (\%) for relaxations of \eqref{eq:stqp-benchmark-model}.}
\label{tab:stqp-relgap-values}\\
\toprule
Family & Method & $n=25$ & $n=50$ & $n=75$ & $n=100$ \\
\midrule
\endfirsthead
\caption[]{continued from previous page}\\
\toprule
Family & Method & $n=25$ & $n=50$ & $n=75$ & $n=100$ \\
\midrule
\endhead
\midrule
\multicolumn{6}{r}{Continued on next page}\\
\endfoot
\bottomrule
\endlastfoot
\multirow{2}{*}{\begin{tabular}[c]{@{}c@{}}COP (easy)\end{tabular}} & SDP & $\infty$ & $\infty$ & $\infty$ & $\infty$ \\*
 & SC-SDP & 10.56 & 10.56 & 10.56 & 10.56 \\
\midrule
\multirow{2}{*}{\begin{tabular}[c]{@{}c@{}}PSD (hard)\end{tabular}} & SDP & 1.80 & 2.21 & 2.33 & 1.42 \\*
 & SC-SDP & 0.00 & 0.00 & 0.09 & 0.02 \\
\midrule
\multirow{2}{*}{\begin{tabular}[c]{@{}c@{}}SPN (hard)\end{tabular}} & SDP & $\infty$ & $\infty$ & $\infty$ & $\infty$ \\*
 & SC-SDP & 0.00 & 0.00 & 0.09 & 0.02 \\
\end{longtable}
\normalsize

Table~\ref{tab:stqp-relgap-values} highlights the advantage of SC-SDP: the standard SDP relaxation is unbounded below for the COP and SPN families, which gives infinite relative gaps, and remains noticeably weaker on the PSD family. In contrast, SC-SDP gives finite gaps on COP instances and is exact or nearly exact on the PSD and SPN instances.

\subsubsection{Presolving comparison}\label{subsubsec:stqp-presolving-comparison}

We further evaluate the presolving performance of SC-SDP on the sparse standard quadratic program \eqref{eq:stqp-benchmark-model}, which includes both the equality constraint $e^\top x=1$ and the nonnegativity constraint $x\ge0$. The test instances are generated from the same COP, PSD, and SPN families used above. 
{To show that the presolving certificate is not tied to the specialized solver RiNNAL-S, we compute the SC-SDP solution and its dual certificate using MOSEK with $\tol=10^{-8}$.} For these instances, the resulting dual solution is already sufficiently informative, so we do not apply the dual-refinement step.

\begin{figure}[ht!]
    \centering
    \begin{subfigure}[t]{0.32\textwidth}
        \centering
        \IfFileExists{pics/stqp_02_relaxation_fix_total.pdf}
        {\includegraphics[width=\textwidth]{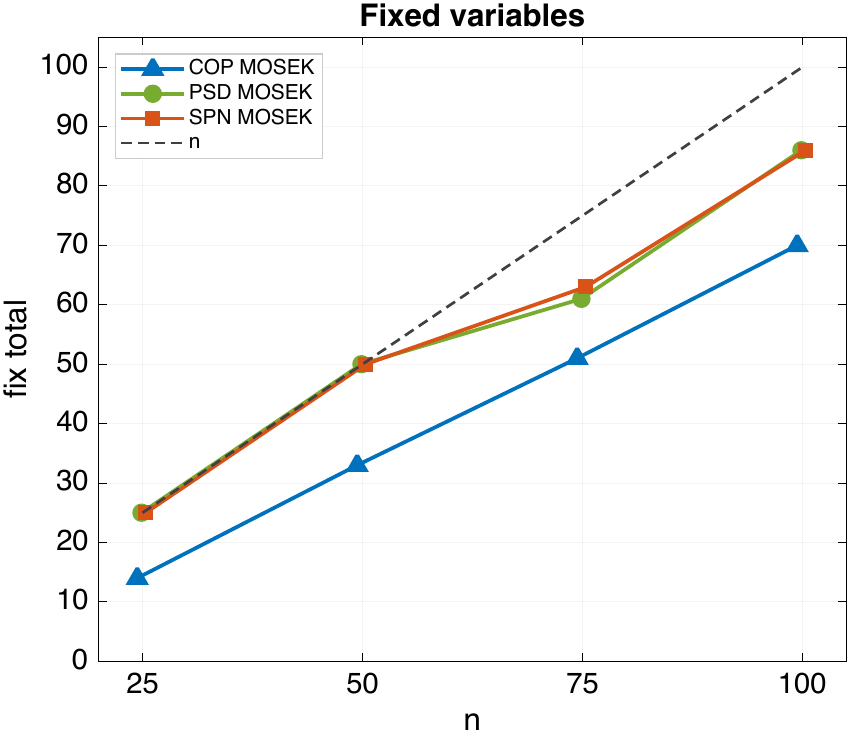}}
        {\fbox{\parbox[c][0.95in][c]{0.9\textwidth}{\centering pending figure}}}
        \caption{Fixed variables.}
        \label{fig:stqp-fixing-fix}
    \end{subfigure}
    \hfill
    \begin{subfigure}[t]{0.32\textwidth}
        \centering
        \IfFileExists{pics/stqp_04_screening_valid_cuts.pdf}
        {\includegraphics[width=\textwidth]{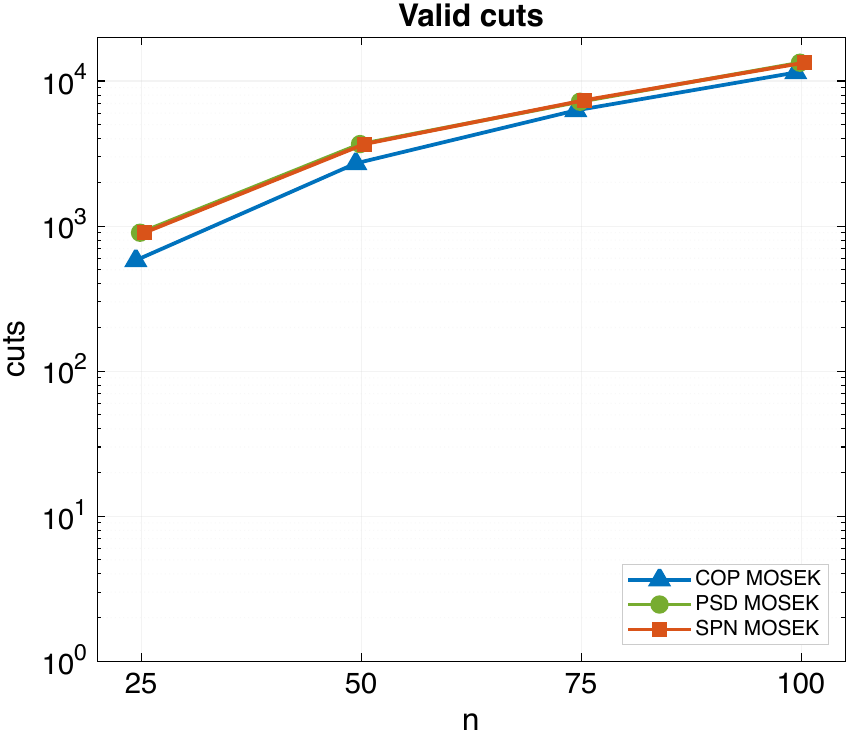}}
        {\fbox{\parbox[c][0.95in][c]{0.9\textwidth}{\centering pending figure}}}
        \caption{Valid cuts.}
        \label{fig:stqp-fixing-valid-cuts}
    \end{subfigure}
    \hfill
    \begin{subfigure}[t]{0.32\textwidth}
        \centering
        \IfFileExists{pics/stqp_05_screening_filtered_cuts.pdf}
        {\includegraphics[width=\textwidth]{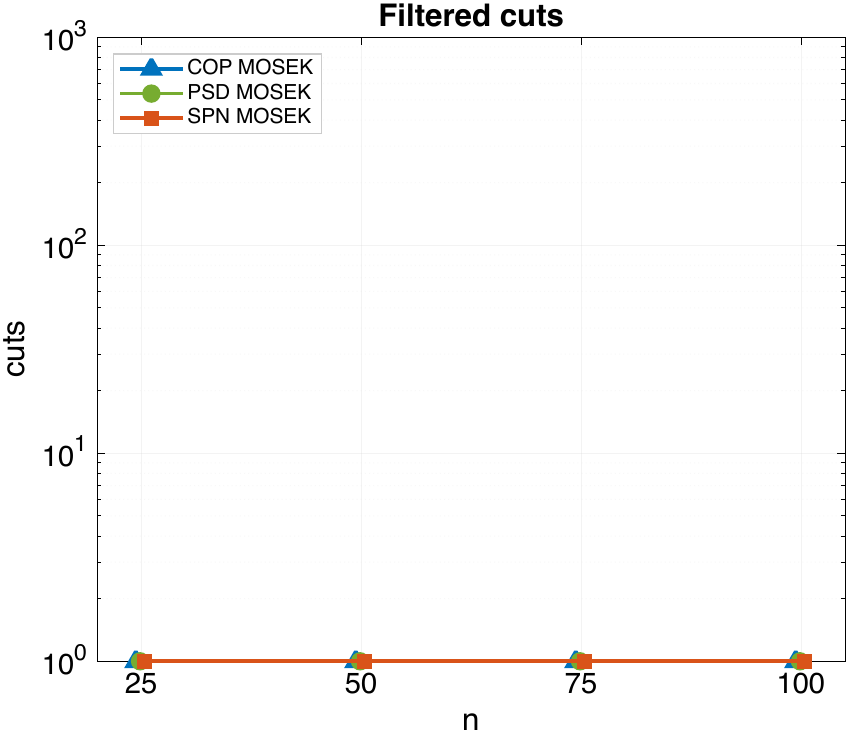}}
        {\fbox{\parbox[c][0.95in][c]{0.9\textwidth}{\centering pending figure}}}
        \caption{Filtered cuts.}
        \label{fig:stqp-fixing-filtered-cuts}
    \end{subfigure}\\[0.8em]
    \begin{subfigure}[t]{0.32\textwidth}
        \centering
        \IfFileExists{pics/stqp_03_relaxation_runtime.pdf}
        {\includegraphics[width=\textwidth]{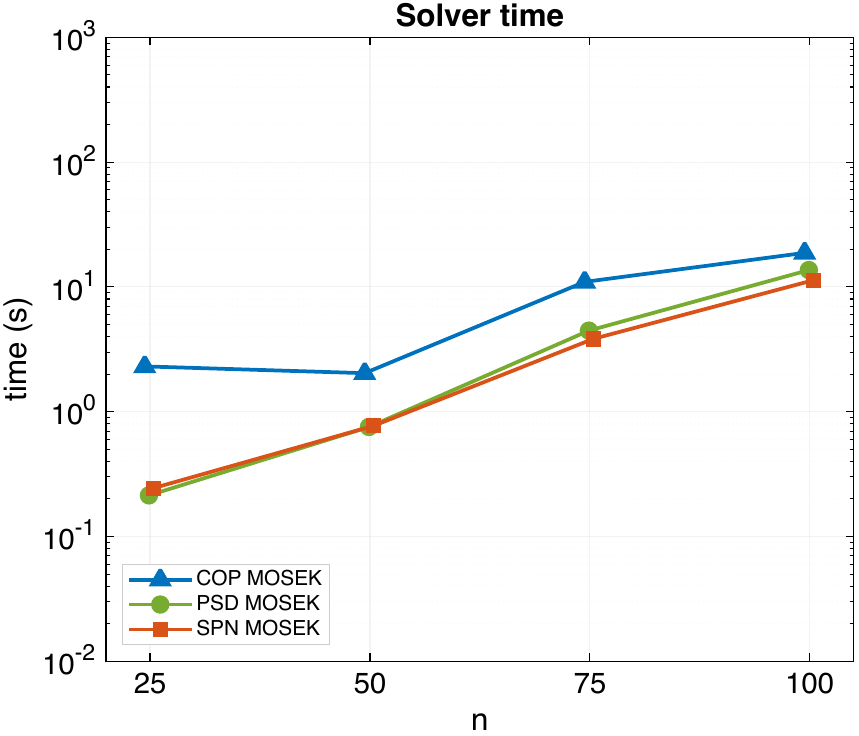}}
        {\fbox{\parbox[c][0.95in][c]{0.9\textwidth}{\centering pending figure}}}
        \caption{Solver time.}
        \label{fig:stqp-fixing-time}
    \end{subfigure}
    \hfill
    \begin{subfigure}[t]{0.32\textwidth}
        \centering
        \IfFileExists{pics/stqp_01_relaxation_relgap.pdf}
        {\includegraphics[width=\textwidth]{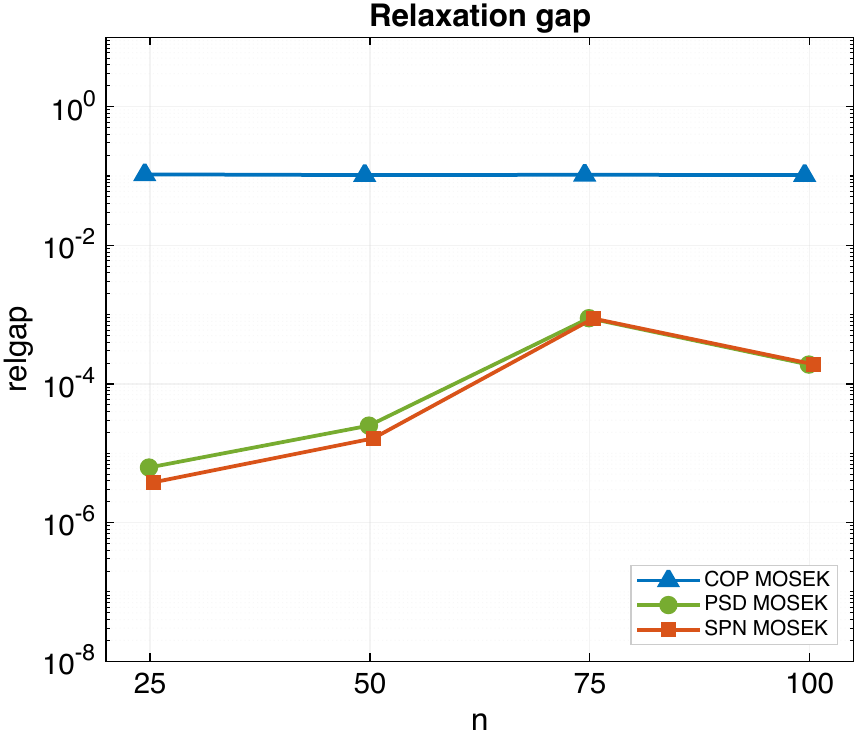}}
        {\fbox{\parbox[c][0.95in][c]{0.9\textwidth}{\centering pending figure}}}
        \caption{Relaxation gap.}
        \label{fig:stqp-fixing-relgap}
    \end{subfigure}
    \hfill
    \begin{subfigure}[t]{0.32\textwidth}
        \centering
        \IfFileExists{pics/stqp_06_gurobi60_relgap.pdf}
        {\includegraphics[width=\textwidth]{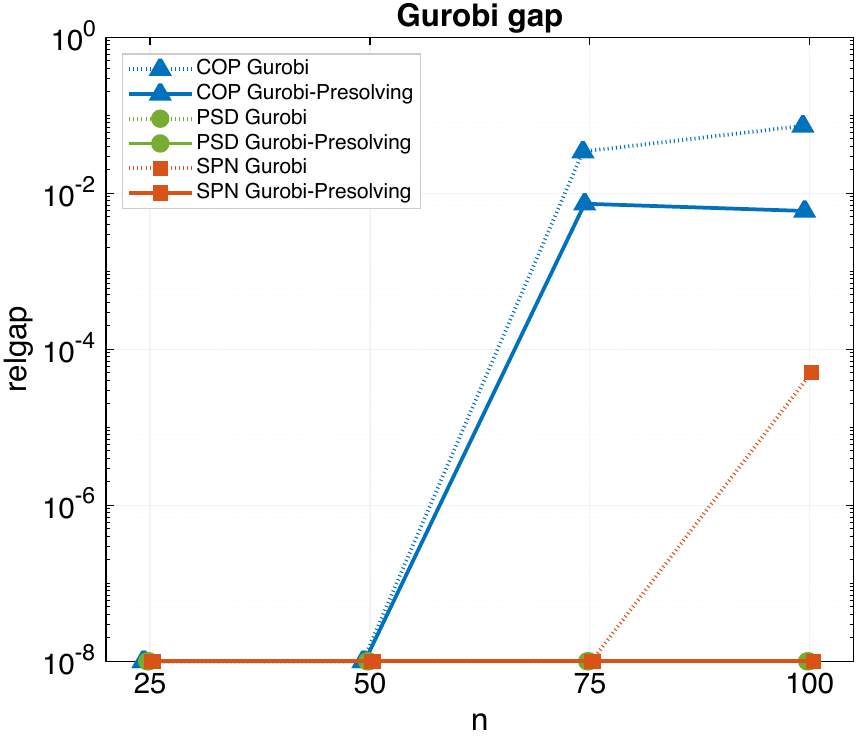}}
        {\fbox{\parbox[c][0.95in][c]{0.9\textwidth}{\centering pending figure}}}
        \caption{Gurobi gap (600s).}
        \label{fig:stqp-fixing-gurobi-gap}
    \end{subfigure}
    \caption{ Presolving and Gurobi performance for \eqref{eq:stqp-benchmark-model}.}
    \label{fig:stqp-fixing}
\end{figure}

Figure~\ref{fig:stqp-fixing} reports the same presolving quantities as in Subsubsection~\ref{subsubsec:srr-presolving-comparison}. Even without dual refinement, the MOSEK certificate yields strong fixing on these constrained sparse QP instances, and this reduction translates into a visibly smaller Gurobi gap within the same time limit. The certificate also produces valid screening cuts, but the filtered-cut count is zero because the variable fixing is already strong enough to imply these pairwise exclusions, leaving no additional nontrivial cuts after filtering.

\subsubsection{RiNNAL-S efficiency}\label{subsubsec:stqp-rinnal-efficiency}

{We next compare RiNNAL-S for solving SC-SDP with reference solvers for solving the SDP--RLT relaxation \eqref{eq:Shor-RLT}.} We compare RiNNAL-S on SC-SDP with the SDP--RLT results reported in \cite{RiNNALplus}. Those results are for the same problem \eqref{eq:stqp-benchmark-model}, but use an SDP--RLT formulation related to \eqref{eq:Shor-RLT} that contains additional RLT constraints and is therefore at least as strong as the SC-SDP relaxation considered here. The SDP--RLT reference models in \cite{RiNNALplus} were solved by RiNNAL+ and SDPNAL+.

\small
\begin{longtable}{llrrccr}
\caption{Computational results for the easy instances of \eqref{eq:stqp-benchmark-model} (COP).}
\label{tab:stqp-solver-results-easy}\\
\toprule
Problem & Algorithm & Iteration & Rank & $R_{\max}$ & Objective & Time \\
\midrule
\endfirsthead
\caption[]{continued from previous page}\\
\toprule
Problem & Algorithm & Iteration & Rank & $R_{\max}$ & Objective & Time \\
\midrule
\endhead
\midrule
\multicolumn{7}{r}{Continued on next page}\\
\endfoot
\bottomrule
\endlastfoot
\multirow{3}{*}{\begin{tabular}[c]{@{}c@{}}COP\\$n=100$\end{tabular}} & RiNNAL-S & 10, 550, 2 & 3   & 5.04e-07 & -1.0195884e-01 & 1.0 \\*
 & RiNNAL+ & 125, 6300, 25 & 102 & 6.64e-07 & -1.0190589e-01 & 30.7 \\
 & SDPNAL+ & 76, 112, 1170 & 102 & 8.59e-07 & -1.0196083e-01 & 40.1 \\
\midrule
\multirow{3}{*}{\begin{tabular}[c]{@{}c@{}}COP\\$n=200$\end{tabular}} & RiNNAL-S & 12, 650, 3 & 3   & 9.63e-07 & -1.0444021e-01 & 2.0 \\*
 & RiNNAL+ & 300, 15050, 60 & 202 & 9.99e-07 & -1.0402158e-01 & 230.5 \\
 & SDPNAL+ & 103, 145, 1511 & 202 & 9.24e-07 & -1.0445084e-01 & 170.7 \\
\midrule
\multirow{3}{*}{\begin{tabular}[c]{@{}c@{}}COP\\$n=500$\end{tabular}} & RiNNAL-S & 15, 800, 3 & 4 & 9.46e-07 & -1.0493236e-01 & 5.4 \\*
 & RiNNAL+ & 672, 33650, 135 & 621 & $1.20\mathrm{e}{-04}^{\dagger}$ & -4.8722936e-01 & 3600.0 \\
 & SDPNAL+ & 200, 227, 2350 & 502 & 9.28e-07 & -1.0466603e-01 & 1779.7 \\
\midrule
\multirow{3}{*}{\begin{tabular}[c]{@{}c@{}}COP\\$n=1000$\end{tabular}} & RiNNAL-S & 20, 1050, 4 & 3 & 8.45e-07 & -1.0530038e-01 & 58.0 \\*
 & RiNNAL+  & - & - & - & - & $>3600.0$  \\
 & SDPNAL+ & - & - & - & - & $>3600.0$ \\
\midrule
\multirow{3}{*}{\begin{tabular}[c]{@{}c@{}}COP\\$n=2000$\end{tabular}} & RiNNAL-S & 37, 1900, 8 & 3 & 4.21e-07 & -1.0542925e-01 & 373.3 \\*
 & RiNNAL+ & - & - & - & - & $>3600.0$ \\
 & SDPNAL+ & - & - & - & - & $>3600.0$ \\
\end{longtable}
\normalsize

Table~\ref{tab:stqp-solver-results-easy} shows that the SC-SDP relaxation solved by RiNNAL-S gives nearly the same objective values as the stronger SDP--RLT reference models solved by RiNNAL+ and SDPNAL+. At the same time, the final ranks produced by RiNNAL-S are much lower than those of the SDP--RLT reference solutions, which further highlights the modeling advantage of SC-SDP. Moreover, RiNNAL-S is roughly 30--700 times faster than RiNNAL+ and 40--300 times faster than SDPNAL+. For the larger COP instances with $n=1000$ and $n=2000$, both SDP--RLT reference solvers exceed the one-hour time limit, whereas RiNNAL-S solves the SC-SDP relaxation in about 1 and 6 minutes, respectively. Overall, these results show that RiNNAL-S keeps essentially the same bound quality while being substantially faster on the easier COP family.

\small
\begin{longtable}{llrrccr}
\caption{Computational results for the hard instances of \eqref{eq:stqp-benchmark-model} (PSD and SPN).}
\label{tab:stqp-solver-results-hard}\\
\toprule
Problem & Algorithm & Iteration & Rank & $R_{\max}$ & Objective & Time \\
\midrule
\endfirsthead
\caption[]{continued from previous page}\\
\toprule
Problem & Algorithm & Iteration & Rank & $R_{\max}$ & Objective & Time \\
\midrule
\endhead
\midrule
\multicolumn{7}{r}{Continued on next page}\\
\endfoot
\bottomrule
\endlastfoot
\multirow{3}{*}{\begin{tabular}[c]{@{}c@{}}PSD\\$n=100$\end{tabular}} & RiNNAL-S & 44, 4400, 1 & 7  & 9.82e-07 & 1.4002781e-02 & 11.7 \\*
 & RiNNAL+ & 374, 18750, 75 & 5  & 8.59e-07 & 1.4174794e-02 & 80.7 \\
 & SDPNAL+ & 801, 1058, 17306 & 2 & 8.50e-07 & 1.4174999e-02 & 468.1 \\
\midrule
\multirow{3}{*}{\begin{tabular}[c]{@{}c@{}}SPN\\$n=100$\end{tabular}} & RiNNAL-S & 119, 6000, 1 & 4  & 8.60e-07 & 1.4007862e-02 & 13.8 \\*
 & RiNNAL+ & 747, 37400, 150 & 107 & 9.31e-07 & 1.3592384e-02 & 204.0 \\
 & SDPNAL+ & 801, 1032, 29187 & 12 & 9.94e-07 & 1.4887842e-02 & 866.1 \\
\midrule
\multirow{3}{*}{\begin{tabular}[c]{@{}c@{}}PSD\\$n=200$\end{tabular}} & RiNNAL-S & 322, 32200, 0 & 44 & 9.82e-07 & 8.9018590e-03 & 137.0 \\*
 & RiNNAL+ & 360, 18050, 72 & 108 & 9.50e-07 & 9.0349366e-03 & 231.0 \\
 & SDPNAL+ & 801, 905, 36057 & 89 & $1.84\mathrm{e}{-06}^{\dagger}$ & 9.3529559e-03 & 3600.0 \\
\midrule
\multirow{3}{*}{\begin{tabular}[c]{@{}c@{}}SPN\\$n=200$\end{tabular}} & RiNNAL-S & 391, 19600, 0 & 8 & 9.28e-07 & 8.9032518e-03 & 74.3 \\*
 & RiNNAL+ & 1812, 90650, 363 & 248 & 9.97e-07 & -1.2530973e-02 & 1556.1 \\
 & SDPNAL+ & 801, 896, 38599 & 25 & $1.02\mathrm{e}{-06}^{\dagger}$ & 1.3190619e-02 & 3600.0 \\
\midrule
\multirow{3}{*}{\begin{tabular}[c]{@{}c@{}}PSD\\$n=500$\end{tabular}} & RiNNAL-S & 347, 34700, 0 & 99 & 9.16e-07 & 2.9152494e-03 & 450.1 \\*
 & RiNNAL+ & 507, 25400, 102 & 234 & 9.88e-07 & 2.9188026e-03 & 1838.0 \\
 & SDPNAL+ & 426, 430, 4726 & 160 & $7.73\mathrm{e}{-06}^{\dagger}$ & 1.0912535e-02 & 3600.0 \\
\midrule
\multirow{3}{*}{\begin{tabular}[c]{@{}c@{}}SPN\\$n=500$\end{tabular}} & RiNNAL-S & 462, 167810, 0 & 15 & 9.78e-07 & 2.9193141e-03 & 3552.5 \\*
 & RiNNAL+ & - & - & - & - & $>3600.0$ \\
 & SDPNAL+ & - & - & - & - & $>3600.0$ \\
\end{longtable}
\normalsize

Table~\ref{tab:stqp-solver-results-hard} reports the harder PSD and SPN families. The comparisons again give nearly identical objective values, confirming that SC-SDP preserves the practical bound quality of SDP--RLT on these instances. These harder families require more iterations and longer runtimes than COP, but RiNNAL-S remains competitive and often much faster than the SDP--RLT reference solvers. In particular, for the SPN instance with $n=500$, both reference solvers exceed the one-hour time limit, while RiNNAL-S solves the SC-SDP relaxation within the limit. Overall, Tables~\ref{tab:stqp-solver-results-easy} and \ref{tab:stqp-solver-results-hard} indicate that SC-SDP achieves a numerical bound similar to that of the stronger SDP--RLT reference relaxation, while being much more efficient to solve.

\subsection{Sparse binary quadratic programs}\label{subsec:numexp-sbqp}
This subsection tests the SC-SDP framework on sparse binary quadratic programs. We consider the following sparse binary quadratic program:
\begingroup
\renewcommand{\theequation}{S-BIQ}
\begin{equation}\label{eq:sbqp-original}    
\min_{x}\;
\left\{
x^\top Q x + 2c^\top x
\;\middle|\;\|x\|_0\le k,\;  x\in\{0,1\}^n
\right\},\tag{S-BIQ}
\end{equation}
\endgroup
where $Q\in\S^n$ and $c\in\R^n$ are given. The binary restriction can be viewed as an additional quadratic equality constraint, since
\[
x\in \{0,1\}^n
\quad\Longleftrightarrow\quad
x_i(1-x_i)=0,\; i\in[n].
\]
In the lifted model, this is simply imposed through the affine equality \(\diag(X)=x\). We also use the valid nonnegative strengthening \(Y\in\mathcal N\), defined in \eqref{eq:nonnegative-cone}, which leads to the SC-SDP relaxation
\begin{equation}\label{eq:sbqp-SDP}
\min_{Y}\;
\Big\{
\langle C,Y\rangle
\;\Big|\; 
\diag(X)=x,\;
Y_{11}=1,\;
Y\in \K\cap \S_+^{n+1}\cap \mathcal{N}
\Big\}.
\end{equation}
We select the test instances from the ORLIB library~\cite{beasley1998heuristic}, maintained by J.E. Beasley, to generate the coefficient matrix $Q$ and vector $c$. We consider four problem dimensions, namely $n=\{250,500,1000,2500\}$, and set the sparsity budget to $k=n/5$. Each dimension includes ten instances, but we report only the first one in this subsection, since the performance on the remaining instances is similar.

\small
\begin{longtable}{llrrccr}
\caption{Computational results for the SC-SDP relaxation of \eqref{eq:sbqp-original}.}
\label{tab:uqp-orlib-results}\\
\toprule
Problem & Algorithm & Iteration & Rank & $R_{\max}$ & Objective & Time \\
\midrule
\endfirsthead
\caption[]{continued from previous page}\\
\toprule
Problem & Algorithm & Iteration & Rank & $R_{\max}$ & Objective & Time \\
\midrule
\endhead
\midrule
\multicolumn{7}{r}{Continued on next page}\\
\endfoot
\bottomrule
\endlastfoot
\multirow{2}{*}{$n = 250$}  & RiNNAL-S & 198, 4040, 52 & 31 & 9.21e-07 & -2.0241801e+04 & 30.3 \\*
                              & MOSEK   & -             & 38 & 6.55e-07 & -2.0241970e+04 & 1033.2 \\
\midrule
\multirow{2}{*}{$n = 500$}  & RiNNAL-S & 130, 2680, 32 & 57 & 2.99e-07 & -5.6538511e+04 & 46.9 \\*
                              & MOSEK   & -             & -  & -        & -               & $>3600.0$ \\
\midrule
\multirow{2}{*}{$n = 1000$} & RiNNAL-S & 189, 3860, 50 & 89 & 8.32e-07 & -1.7015018e+05 & 325.0 \\*
                              & MOSEK   & -             & -  & -        & -               & $>3600.0$ \\
\midrule
\multirow{2}{*}{$n = 2500$} & RiNNAL-S & 198, 4040, 57 & 167 & 6.61e-07 & -7.0768612e+05 & 2768.4 \\*
                              & MOSEK   & -             & -  & -        & -               & $>3600.0$ \\
\end{longtable}
\normalsize

Table~\ref{tab:uqp-orlib-results} shows that RiNNAL-S remains highly effective for the SC-SDP relaxation of \eqref{eq:sbqp-original}: it consistently solves all tested instances within the prescribed tolerance, whereas solving the reduced SDP--RLT relaxation \eqref{eq:Decomposed-SDP-z} by MOSEK
already becomes impractical on the larger instances and exceeds the one-hour time limit for $n\geq 500$. Overall, these results indicate that the proposed algorithm extends effectively to the sparse binary setting while still delivering strong computational performance.

\subsection{Sparse quadratic knapsack problems}
\label{subsec:numexp-sqkp-conflict}

We finally consider a more complicated QCQP to test SC-SDP beyond the unconstrained and simplex-constrained settings. Specifically, we study the sparse quadratic knapsack problem with multiple knapsack constraints and pairwise conflict constraints, which takes the form
\begingroup
\renewcommand{\theequation}{S-QKP}
\begin{equation}\label{eq:sqkp-conflict-model}
\min_{x}
\left\{
x^\top Qx + 2c^\top x
\;\middle|\;
Gx\le d,\;
x_i x_j=0\;\; (i,j)\in E,\;
\ x\in\{0,1\}^n,\;
\|x\|_0\le k
\right\},
\tag{S-QKP}
\end{equation}
\endgroup
where $G\in\R^{p\times n}$ and $d\in\R^p$ define multiple knapsack constraints, and $E$ encodes pairwise conflicts. The constraint $x_i x_j=0$ prevents two incompatible binary variables from being selected simultaneously. 
In the lifted model, the binary constraints and pairwise conflicts are handled as ordinary affine constraints by adding \(\diag(X)=x\) and \(X_{ij}=0\) for \((i,j)\in E\), respectively.

We randomly generate the profit and weight data following the classical
quadratic knapsack procedure of Gallo et al.~\cite{gallo1980quadratic}, which
has been widely used in the literature
\cite{billionnet2004exact,caprara1999exact,pisinger2007quadratic,tang2024feasible}.
The multiple knapsack and conflict constraints are added to test the SC-SDP
relaxation on a richer binary QCQP, in the spirit of generalized and
disjunctively constrained knapsack models
\cite{saracc2014generalized,yamada2002heuristic}. We use \(p=10\) knapsack
constraints, \(k=\lceil 0.2n\rceil\), and quadratic density \(0.25\). Item
profits \(c_j\) and nonzero pairwise profits \(q_{ij}=q_{ji}\) are integers
uniformly drawn from \(\{1,\ldots,100\}\), where each pair \(i<j\) is selected
with probability \(0.25\). Knapsack weights \(G_{rj}\) are integers uniformly
drawn from \(\{1,\ldots,50\}\), and capacities are set to
\(d_r=\lceil 0.5\sum_j G_{rj}\rceil\). Finally, each conflict edge is sampled independently with
density \(\rho\in\{10^{-2},10^{-1}\}\). MOSEK is not reported for this
benchmark: it already encountered memory limitations at the smallest
reported size \(n=500\), while its runtime on smaller pilot instances
exceeded the one-hour time limit. 

\small
\begin{longtable}{crlrrccr}
\caption{Computational results for the SC-SDP relaxation of \eqref{eq:sqkp-conflict-model}.}
\label{tab:sqkp-conflict-results}\\
\toprule
$\rho$ & $n$ & Algorithm & Iteration & Rank & $R_{\max}$ & Objective & Time \\
\midrule
\endfirsthead
\caption[]{continued from previous page}\\
\toprule
$\rho$ & $n$ & Algorithm & Iteration & Rank & $R_{\max}$ & Objective & Time \\
\midrule
\endhead
\midrule
\multicolumn{8}{r}{Continued on next page}\\
\endfoot
\bottomrule
\endlastfoot
$10^{-1}$ & 1000 & RiNNAL-S & 78, 870, 16 & 198 & 9.33e-07 & -1.0932744e+05 & 142.9 \\*
$10^{-1}$ & 2000 & RiNNAL-S & 108, 1170, 22 & 388 & 6.40e-07 & -2.2755262e+05 & 940.1 \\*
$10^{-1}$ & 3000 & RiNNAL-S & 102, 1110, 21 & 580 & 9.95e-07 & -3.4644352e+05 & 3196.9 \\
\midrule
$10^{-2}$ & 500 & RiNNAL-S & 789, 7980, 158 & 33 & 9.99e-07 & -1.5698783e+05 & 536.7 \\*
$10^{-2}$ & 1000 & RiNNAL-S & 360, 3690, 72 & 52 & 9.64e-07 & -4.9705848e+05 & 616.8 \\*
$10^{-2}$ & 1500 & RiNNAL-S & 408, 4170, 82 & 76 & 9.00e-07 & -9.5741504e+05 & 1339.6 \\
\end{longtable}
\normalsize

Table~\ref{tab:sqkp-conflict-results} shows that RiNNAL-S remains effective on
this more general constrained binary QCQP. For \(\rho=10^{-1}\), the reported
instances with \(n=1000,2000,3000\) are all solved within the one-hour time limit; for \(\rho=10^{-2}\),
it still solves the reported instances up to \(n=1500\) within the time limit. These instances are
substantially more complex, since they include multiple knapsack
constraints and pairwise conflict constraints in addition to the sparsity
budget. The final ranks are not extremely small, especially for
\(\rho=10^{-1}\), but RiNNAL-S still reaches the prescribed accuracy and remains
scalable. This demonstrates that the SC-SDP framework and the RiNNAL-S algorithm can be effective for broader classes of sparse QCQPs.

\section{Conclusions}

This paper develops a sparsity-cone semidefinite programming (SC-SDP) framework for sparse quadratic optimization. The proposed relaxation replaces the conventional lifted formulation involving both continuous and binary variables by a compact SDP in the original lifted matrix of size $n+1$, together with a single structured constraint involving the sparsity-cone $\K$. We prove that this compact formulation is equivalent in strength to the SDP--RLT relaxation, while substantially reducing the SDP dimension and simplifying the constraint structure. We also analyze the cones $\mathcal K$ and $\mathcal K^*$, derive an efficient projection operation for $\mathcal K$, and use the SC-SDP dual structure to obtain support-restriction, variable fixing, screening-cut, and dual-refinement certificates.

For computation, we develop RiNNAL-S, a two-phase augmented Lagrangian method that combines a low-rank phase with a convex lifting phase and exploits the structured projections associated with $\mathcal F$ and $\mathcal K$. Numerical experiments on projection subproblems and several classes of sparse quadratic programs show that SC-SDP preserves the bound quality of SDP--RLT while offering significant computational advantages, supporting effective presolving, and often producing much lower-rank solutions. Future work includes integrating these certificates more fully into branch-and-bound algorithms and further exploiting the sparsity-cone geometry for broader sparse QCQP models.

\bibliographystyle{abbrv}
\bibliography{SCSDP}

\clearpage
\appendix

\section{Support restrictions and strong branching}\label{app:support-restrictions}

This section records the restricted-relaxation viewpoint mentioned in Section~\ref{sec:presolving}. Let $S\subseteq[n]$ be the indices fixed to $z_i=1$,  $N\subseteq[n]$  the indices fixed to $z_i=0$, and  $R:=[n]\setminus(S\cup N)$ the remaining indices. In the unconstrained case \eqref{eq:QP}, the corresponding support-restricted subproblem is
\[
\min_x\left\{x^\top Qx+2c^\top x
\;\middle|\;
x_i=0\ \forall i\in N,\;
\|x_R\|_0\le k-|S|
\right\}.
\]
The SC-SDP relaxation is obtained by applying the sparsity-cone constraint only to the active block. Writing $Y=[1,x^\top;x,X]$, this gives
\[
\min\left\{\langle C,Y\rangle
\;\middle|\;
Y_{11}=1,\;
Y\succeq 0,\;
X_{ii}=0\ \forall i\in N,\;
\begin{bmatrix}
(k-|S|)Y_{11} & Y_{1R}\\
Y_{R1} & \Diag(Y_{RR})
\end{bmatrix}\succeq 0
\right\}.
\]
For problems with additional linear constraints, the same construction applies after updating the affine constraints and the active index set. Thus the cone geometry and the projection routines used by RiNNAL-S can be reused at such restricted nodes.

In the unconstrained sparse ridge regression setting, the tightness results in Proposition~\ref{thm:SDP-perspective} and Theorem~\ref{thm:srr-opt-persp} apply directly after the linear restrictions $z_i=1$ for $i\in S$ and $z_i=0$ for $i\in N$. Thus, with the obvious notation, $v^{\mathrm{PR}}_{S,N}\le v^{\mathrm{SC}}_{S,N}=v^{\mathrm{OP}}_{S,N}$. If $v^{\mathrm{SC}}_{S,N}$ exceeds a known upper bound, then the pattern $(S,N)$ cannot occur in an optimal solution, equivalently the no-good cut $\sum_{i\in S}z_i+\sum_{i\in N}(1-z_i)\le |S|+|N|-1$ is valid for all optimal solutions. Strong branching is the special case $S=\emptyset,N=\{p\}$ or $S=\{p\},N=\emptyset$.

\section{Additional projection experiments}\label{app:additional-projection}

This section reports the supplementary projection experiments for $\K\cap\mathcal N$, with \(\mathcal N\) defined in \eqref{eq:nonnegative-cone}, mentioned in Subsection~\ref{subsec:projection-efficiency}. These tests use the same table notation as Table~\ref{tab:projection-benchmark-direct}. The instances for $\K\cap\mathcal{N}$ are generated in the same manner as those for $\K$, except that the feasible seed is adapted to the nonnegative setting: we impose $x\geq 0$ and generate the off-diagonal entries of the lower-right block to be nonnegative before perturbation.

\small
\begin{longtable}{lrrrrr}
\caption{Additional projection results on $\K\cap\mathcal{N}$.}
\label{tab:projection-benchmark-additional}\\
\toprule
Set & $n$ & RiNNAL-S (s) & MOSEK (s) & Speedup & RelErr \\
\midrule
\endfirsthead
\caption[]{continued from previous page}\\
\toprule
Set & $n$ & RiNNAL-S (s) & MOSEK (s) & Speedup & RelErr \\
\midrule
\endhead
\midrule
\multicolumn{6}{r}{Continued on next page}\\
\endfoot
\bottomrule
\endlastfoot
$\K\cap\mathcal{N}$ & 100   & 0.0006 & 17.5666  & 27629    & 3.35e-06 \\
$\K\cap\mathcal{N}$ & 200   & 0.0015 & 458.0648 & 306603   & 5.02e-07 \\
$\K\cap\mathcal{N}$ & 1000  & 0.0228 & $>3600.0$  & --       & -- \\
$\K\cap\mathcal{N}$ & 10000 & 5.0250 & $>3600.0$  & --       & -- \\
\end{longtable}
\normalsize

The results are consistent with the discussion in Subsection~\ref{subsec:projection-efficiency}. The results for $\K\cap\mathcal N$, where \(\mathcal N\) is defined in \eqref{eq:nonnegative-cone}, closely mirror those for $\K$, as expected from the shared one-dimensional reduction.

\end{document}